\documentclass[a4paper,10pt]{article}
\usepackage{fullpage}
\usepackage{amsmath,amssymb,latexsym,amsthm}
\usepackage[pdftex]{graphicx}
\usepackage[utf8]{inputenc}
\usepackage{color}

\newcommand{\R}{\mathbb{R}}

\newcommand{\N}{\mathbb{N}}

\newcommand{\T}{\mathbb{T}}
\renewcommand{\leq}{\leqslant}
\renewcommand{\geq}{\geqslant}
\renewcommand{\div}{\operatorname{div\,}}
\newcommand{\curl}{\operatorname{curl}}

\newcommand{\tr}{\operatorname{tr}}
\newcommand{\dist}{\operatorname{dist}}
\newcommand{\diam}{\operatorname{diam}}

\newcommand{\Id}{\operatorname{Id}}

\newtheorem{Theorem}{Theorem}
\newtheorem{Definition}{Definition}
\newtheorem{Corollary}{Corollary}
\newtheorem{Proposition}{Proposition}
\newtheorem{Lemma}{Lemma}
\newtheorem{Remark}{Remark}

\title{Uniqueness results for weak solutions \\ of two-dimensional fluid-solid systems}
\author{Olivier Glass\footnote{CEREMADE, UMR CNRS 7534,
Universit\'e Paris-Dauphine, 
Place du Mar\'echal de Lattre de Tassigny,
75775 Paris Cedex 16,
FRANCE
},
Franck Sueur\footnote{CNRS, UMR 7598, Laboratoire Jacques-Louis Lions, F-75005, Paris, France}
\footnote{UPMC Univ Paris 06, UMR 7598, Laboratoire Jacques-Louis Lions, F-75005, Paris, France}}
\begin{document}
\maketitle
\begin{abstract}
In this paper, we consider two systems modelling the evolution of a rigid body in an incompressible fluid  in a bounded domain of the plane. 
The first system corresponds to an  inviscid fluid  driven by the Euler equation whereas the other one corresponds to a viscous fluid 
driven by the Navier-Stokes system. In both cases we investigate the uniqueness of weak solutions, \`a la Yudovich for the Euler case, \`a la Leray for the Navier-Stokes case, as long as no collision occurs.
\end{abstract}
%
%
%
%%%%%%%%%%%%%%%%%%%%%%%%%%%%%%%%%%%%%%%%%%%%%%%%%%%%%%
%
%
\section{Introduction}
In this paper, we consider two systems modelling the evolution of a rigid body in an incompressible fluid in dimension two. The two cases correspond respectively to the inviscid case, where the fluid is driven by the Euler equation, and the viscous case, where it is driven by the Navier-Stokes system. In both cases we investigate the uniqueness of weak solutions (\`a la Yudovich for the Euler case, \`a la Leray for the Navier-Stokes case). \par
\ \par
To model the body-fluid systems, we introduce the following objects. Let $\Omega$ a smooth connected bounded open set in $\R^{2}$, and ${\mathcal S}_{0}$ smooth closed connected and simply connected subset of $\Omega$. We consider the motion in the domain $\Omega$ of a solid occupying at time $t$ the domain $\mathcal{S} (t) \subset \Omega$, where $\mathcal{S} (0)=\mathcal{S}_{0}$. \par
\ \par
The motion of this solid is rigid, so that ${\mathcal S}(t)$ is obtained by a rigid movement (that is a translation and a rotation) from its initial position ${\mathcal S}_{0}$. The group of rigid transformations of the plane is the special Euclidean group, denoted by $SE(2)$. We will denote $m>0$ and $ \mathcal{J}>0$ respectively the mass and the inertia of the body and $h(t)$ the position of its center of mass at time $t$. We also introduce
$$\ell (t) := h' (t),$$
the velocity of the center of mass and
$$r  (t):= \theta'(t),$$
the angular velocity of the body. The angle $\theta$ measures the rotation between ${\mathcal S}(t)$ and ${\mathcal S}_{0}$.
Accordingly, the solid velocity is given by
\begin{equation} \label{vietendue}
u_{{\mathcal S}}(t,x) := \ell(t) + r(t) (x-h(t))^{\perp}.
\end{equation}
A way to represent the rigid motion from ${\mathcal S}_{0}$ to $\mathcal{S}(t)$ is to introduce the rotation matrix 
\begin{equation*}
Q (t):= 
\begin{bmatrix}
\cos  \theta (t) & - \sin \theta (t)
\\  \sin  \theta (t) & \cos  \theta (t)
\end{bmatrix}.
\end{equation*}
Then the position $\tau (t,x) \in \mathcal{S} (t)$  at 
the time $t$ of the point fixed to the body with an initial position $x$ is 
\begin{equation*}
\tau(t,x) := h (t) + Q (t)(x- h(0)),
\end{equation*}
so that
\begin{equation*}
{\mathcal S}(t) = \tau (t)({\mathcal S}_{0}).
\end{equation*}
We denote 
\begin{equation} \label{Solideci}
h (0)= h_0 , \ h' (0)= \ell_0, \ \theta (0) =  0 ,\ r  (0)=  r _0,
\end{equation}
the initial values of the solid data. \par

Let us stress here that, given some initial data, it suffices to know $(\ell,r)$ to deduce all the other objects above, since 
to $(\ell,r) \in C^{0}([0,T];\R^{2} \times \R)$ we can associate $(h^{\ell,r} , \theta^{\ell,r}) \in C^{1}([0,T]; \R^{2} \times \R)$ by
\begin{equation} \label{Eq:xbq}
h^{\ell,r} (t) = h_{0} + \int_{0}^{t} \ell, \quad 
\theta^{\ell,r} (t) =   \int_{0}^{t} r ,
\end{equation}
the velocity
\begin{equation} \label{Defvsolide}
u_{{\mathcal S}}^{\ell,r}(t,x) := \ell(t) + r(t) (x-h^{\ell,r}(t))^\perp,
\end{equation}
and
\begin{eqnarray}
\label{Eq:rota}
Q^{\ell,r} (t):= 
\begin{bmatrix}
\cos  \theta^{\ell,r} (t) & - \sin \theta^{\ell,r} (t)
\\  \sin  \theta^{\ell,r} (t) & \cos  \theta^{\ell,r} (t)
\end{bmatrix}.
\end{eqnarray}
We also deduce the rigid displacement and the position of the solid, let us say $\tau^{\ell,r}(t)$ and ${\mathcal S}^{\ell,r}(t)$ defined by
\begin{equation} \label{Eq:St}
\tau^{\ell,r}(t): x \mapsto Q^{\ell,r}(t)[x- h_{0}] + h^{\ell,r} (t) \in SE(2), \ \text{ and } \  {\mathcal S}^{\ell,r}(t)=\tau^{\ell,r}(t){\mathcal S}_{0}.
\end{equation}
Then we define the fluid domain as 
\begin{equation} \label{Eq:Do}
{\mathcal F}^{\ell,r}(t):=\Omega \setminus {\mathcal S}^{\ell,r}(t) .
 \end{equation}
We may omit the dependence on $(\ell,r)$ when there is no ambiguity on the various objects defined above. \par
\ \par
In the rest of the domain, that is in the open set 
\begin{equation*}
\mathcal{F}(t) := \Omega \setminus {\mathcal S} (t),
\end{equation*}
evolves a planar fluid driven by the Euler or the Navier-Stokes equations. We denote correspondingly 
\begin{equation*}
{\mathcal F}_{0}:= \Omega \setminus {\mathcal S}_{0},
\end{equation*}
the initial fluid domain. We will consider for each $t$ the velocity field $u=u(t,x) \in \R^{2}$ and the pressure field $p=p(t,x) \in \R$ in ${\mathcal F}(t)$. The fluid will be supposed in both cases to be homogeneous of density $1$, in order to simplify the equations (and without loss of generality). We denote
\begin{equation}\label{Eulerci2}
u |_{t= 0} = u_0 ,
\end{equation}
the initial value of the fluid velocity field. \par
\par
Now to be more specific on the systems under view, we distinguish between the two cases.
\subsection{The Euler case}
In this case, the fluid equation is the incompressible Euler equation and the body evolves according to Newton's law, under the influence of the pressure alone. The boundary conditions correspond to the impermeability of the boundary and involve the normal component of the velocity.
The complete system driving the dynamics reads
\begin{gather}
\label{Euler1}
\frac{\partial u}{\partial t}+(u\cdot\nabla)u + \nabla p =0 \ \text{ for } \ x \in \mathcal{F}(t), \\
\label{Euler2}
\div u = 0 \ \text{ for } \  x \in \mathcal{F}(t) ,  \\
\label{Euler3}
u \cdot n =  u_\mathcal{S} \cdot n \ \text{ for } \ x\in \partial \mathcal{S}  (t),  \\
\label{Euler4}
u \cdot n =  0 \ \text{ for } \ x \in \partial \Omega,  \\
\label{Solide1} 
m  h'' (t) =  \int_{ \partial \mathcal{S} (t)} p \, n \, d \sigma ,  \\
\label{Solide2} 
\mathcal{J}  \theta'' (t) =    \int_{ \partial   \mathcal{S} (t)} p \, (x-  h (t) )^\perp  \cdot n \, d \sigma.
\end{gather}
When $x=(x_1,x_2)$ the notation $x^\perp$ stands for  $$x^\perp :=( -x_2 , x_1 ),$$ $n$ denotes the unit outward normal on $\partial \mathcal{F}(t)$, $d \sigma$ denotes the integration element on the boundary  $\partial \mathcal{S}(t)$ of the body. \par
For this system, one can prove the existence of a solution ``\`a la Yudovich'' \cite{Yudovich} on a time interval limited only by the possible encounter of the solid and the boundary $\partial \Omega$. The main assumption is that the initial vorticity
\begin{equation*}
\omega_{0} := \curl u_{0},
\end{equation*}
is bounded in $\Omega$.
\begin{Theorem} \label{ThmYudo}
For any $u_0 \in C^{0}(\overline{\mathcal{F}_0};\R^{2})$, $(\ell_0,r_0) \in \R^2 \times \R$, such that:
\begin{equation} \label{CondCompatibilite}
\div u_0 =0 \text{ in } {\mathcal F}_0, \   u_0   \cdot  n = (\ell_0 + r_0 (x-h_{0})^{\perp})   \cdot  n \text{ on } \partial \mathcal{S}_0, \ 
 u_0   \cdot  n = 0 \text{ on } \partial \Omega,
\end{equation}
and
\begin{equation} \label{TourbillonYudo}
\omega_{0} := \curl u_0  \in L^{\infty}({\mathcal F}_0),
\end{equation}
there exists $T>0$ and a solution 
\begin{equation*}
(\ell,r,u) \in  C^1 ([0,T]; \R^2 \times \R) \times [L^{\infty}(0,T; \mathcal{LL}({\mathcal F}(t))) \cap C^{0}([0,T]; W^{1,q}({\mathcal F}(t)))], \ \ \forall q \in [1,+\infty),
\end{equation*}
of \eqref{Euler1}-\eqref{Solide2}. Moreover, if $T<+\infty$ is maximal, then
\begin{equation} \label{Explosion}
\dist ({\mathcal S}(t), \partial \Omega) \rightarrow 0 \text{ as } t \rightarrow T^{-}.
\end{equation}
\end{Theorem}
Several comments are in order here.
\par
First the notation $\mathcal{LL}({\mathcal F}(t))$ refers to the space of log-Lipschitz functions on ${\mathcal F}(t)$, that is the set of functions $f \in L^{\infty}({\mathcal F}(t))$ such that
\begin{equation} \label{DefLL}
\| f \|_{\mathcal{LL}({\mathcal F}(t))} := \| f\|_{L^{\infty}({\mathcal F}(t))} + \sup_{x\not = y} \frac{|f(x)-f(y)|}{|x-y|(1+ \ln^{-}|x-y|)} < +\infty.
\end{equation}
For a functional space $X$ of functions depending on the variable $x$, the notation $L^\infty (0,T; X(\mathcal{F}(t)))$ or $C([0,T]; X (\mathcal{F}(t)))$ stands for the space of functions defined for each $t$ in the fluid domain ${\mathcal F}(t)$, and which can be extended to functions in $L^\infty (0,T; X(\R^{2}))$ or $C([0,T];X(\R^{2}))$ respectively. In the same spirit, we will make the abuse of notations $[0,T] \times {\mathcal F}(t)$ for $\cup_{t \in [0,T]} \{ t \} \times {\mathcal F}(t)$. \par
The other remark is that the pressure $p$ is uniquely defined, up to a function depending only on time, by  $(\ell,r,u)$
as a function of $L^{\infty}(0,T; H^{1}({\mathcal F}(t)))$  (see Corollary  \ref{CorAPrioriEstimates}). 
In particular this gives a sense to the right hand sides of \eqref{Solide1}-\eqref{Solide2}.
\par
\ \par
In the case when $\Omega= \R^{2}$, the equivalent of Theorem \ref{ThmYudo} (together with the uniqueness in this particular case), was proven in \cite{smallbody}. In this particular situation, one can make a {\it rigid} change of variable to write the system in ${\mathcal F}_{0}$, which simplifies the analysis. \par
We provide in the appendix a proof of Theorem \ref{ThmYudo} in the case considered here where the system occupies a bounded domain.
\par
\ \par
The first main result of this paper is  the following.
\begin{Theorem} \label{UniqYudo}
The above solution is unique in its class.
\end{Theorem}
\subsection{The Navier-Stokes case}
We now turn to the case of a viscous fluid. \par
In this case, the fluid equation is the incompressible Navier-Stokes equation and the body evolves according to Newton's law, under the influence of the whole Cauchy stress tensor. The boundary conditions are the usual no-slip condition for the velocity field.
The complete system driving the dynamics reads
\begin{gather}
\label{NS1}
\frac{\partial u}{\partial t}+(u\cdot\nabla)u -  \Delta u + \nabla p =0 \ \text{ for } \ x \in \mathcal{F}(t), \\
\label{NS2}
\div u = 0 \ \text{ for } \  x \in \mathcal{F}(t) ,  \\
\label{NS3}
u =  u_\mathcal{S} \ \text{ for } \ x\in \partial \mathcal{S}  (t),  \\
\label{NS4}
u =  0 \ \text{ for } \ x\in \partial \Omega,  \\
\label{Solide1NS}
m  h'' (t) = - \int_{ \partial \mathcal{S} (t)} \T n \, d \sigma ,  \\
\label{Solide2NS}
\mathcal{J}  \theta'' (t) =   - \int_{ \partial   \mathcal{S} (t)} \T  n \cdot (x-  h (t) )^\perp \, d \sigma,
\end{gather}
where the same notations for $x$, $d \sigma$ and  $h$ are used as in the previous paragraph, and where
\begin{equation*}
\T(u,p)  := - p \Id +  2 D u \ \text{ with } Du := \frac{1}{2} (\nabla u + \nabla u^{T}). 
\end{equation*}
For this system, one can prove the existence of a weak solution ``\`a la Leray'' \cite{LerayActa,LerayJMPA}, for which the main assumption is that the initial velocity $u_{0}$ is square-integrable.
To define more precisely what we mean by a weak solution of \eqref{NS1}-\eqref{Solide2NS}, let us define a velocity field globally on $\Omega$ by setting 
\begin{equation} \label{globally}
\overline{u} (t,x) := u(t,x) \ \text{ for } \ x \in {\mathcal F}(t) 
\ \text{ and } \ 
\overline{u} (t,x) := u_\mathcal{S} (t,x) \ \text{ for } \ x \in {\mathcal S}(t),
\end{equation}
where $u_\mathcal{S}$ is given by \eqref{vietendue}. 
We will say that $\overline{u}$ is {\it compatible} with $(\ell,r)$ when $\overline{u}(t,\cdot)$ belongs to $H^{1}(\Omega)$ for almost every $t$ and \eqref{globally} holds with ${\mathcal F}(t)={\mathcal F}^{\ell,r}(t)$ and $u_\mathcal{S}$ is given by \eqref{vietendue}.
Similarly, for the initial data, we define a velocity field $\overline{u}_0$ by setting 
\begin{equation*}
\overline{u}_0 (x) := u_0 (x) \ \text{ for } \ x\in {\mathcal F}_0 \ \text{ and } \ \overline{u}_0 (x) := \ell_0 + r_0 (x-h_{0})^{\perp} \ \text{ for } \ x \in {\mathcal S}_0.
\end{equation*}
Now to define the notion of weak solutions that we consider, it will be useful to introduce the density inside the solid at initial time $t=0$ as the function $\rho_{{\mathcal S}_{0}} (x) $, for $x \in \mathcal{S}_0$. Accordingly, the mass and the inertia of the solid satisfy
\begin{equation*}
m = \int_{{\mathcal S}_{0}} \rho_{{\mathcal S}_{0}}(x) \, dx
\ \text{ and } \ 
{\mathcal J} = \int_{{\mathcal S}_{0}} \rho_{{\mathcal S}_{0}}(x) |x - h_{0}|^{2} \, dx.
\end{equation*}
We extend this initial density as a function on the whole domain $\Omega$ by setting
\begin{equation} \label{Eq:Rho0}
\rho_{0}(x) = \rho_{{\mathcal S}_{0}} \text{ in } {\mathcal S}_{0}
\ \text{ and } \ 
\rho_{0}(x) = 1 \text{ in } {\mathcal F}_{0}.
\end{equation}
Given a rigid movement $(\ell,r)$, we define the solid density as:
\begin{equation} \label{Eq:RhoS}
\rho_{{\mathcal S}}(t,x) = \rho_{{\mathcal S}_{0}} ( (\tau^{\ell,r} (t,\cdot))^{-1}  (x) ) \text{ in } {\mathcal S}^{\ell,r}(t)
\ \text{ and } \ 
\rho_{{\mathcal S}(t)}(x) = 0 \text{ in } {\mathcal F}^{\ell,r}(t),
\end{equation}
and the density $\rho(t,x)$ in $[0,T] \times \Omega$ as 
\begin{equation} \label{Eq:Rho}
\rho (t,x) = \rho_{{\mathcal S}}(t,x) \text{ in } {\mathcal S}^{\ell,r}(t)
\ \text{ and } \ 
\rho(x) = 1 \text{ in } {\mathcal F}^{\ell,r}(t).
\end{equation}
\begin{Definition}[see \cite{GLS,DE-CPDE,Conca1,SMST,FE1}] \label{DefSolusLeray}
We say that 
\begin{equation*}
(\ell,r,\overline{u}) \in  C^0 ([0,T]; \R^2 \times \R) \times [L^{\infty} (0,T; L^2 (\Omega)) \cap L^2 (0,T; H^{1}(\Omega))]
\end{equation*}
is a weak solution of \eqref{NS1}-\eqref{Solide2NS} with the initial data \eqref{Solideci}-\eqref{Eulerci2} if $\overline{u}$ is divergence free, 
\begin{equation} \label{Compatibilite}
\overline{u} \ \text{  is compatible with } (\ell,r),
\end{equation}
and for any divergence free vector field $\phi \in C^\infty_{c} ([0,T] \times \Omega ;\R^2 )$ such that $D\phi (t,x)= 0 $ when $t \in [0,T]$ and $x \in \mathcal{S}^{\ell,r} (t)$, there holds, when $\rho$ is given by \eqref{Eq:Rho}:
\begin{equation} \label{weakformulation}
\int_{\Omega }  \rho_{0}  \overline{u}_{0} \cdot  \phi \vert_{t=0} - \int_{\Omega }  (\rho  \overline{u} \cdot  \phi )\vert_{t=T}
+ \int_{(0,T) \times \Omega } \rho  \overline{u} \cdot \frac{\partial \phi}{\partial t}   +  (\overline{u}\otimes \overline{u}  -  2 D\overline{u} )  :  D\phi =0 .
\end{equation}
We will also say that $(\ell,r,u) \in  C^0 ([0,T]; \R^2 \times \R) \times [L^{\infty} (0,T; L^2 ( {\mathcal F}^{\ell,r}(t) )) \cap L^2 (0,T; H^{1}( {\mathcal F}^{\ell,r}(t) ))]$ is a solution when $(\ell,r,\overline{u})$ with $\overline{u}$ defined by \eqref{globally} is a solution. \par
\end{Definition}
\ \par
Now we have the following existence theorem of weak solutions.
\begin{Theorem}[see \cite{GLS,DE-CPDE,Conca1,SMST,FE1}] 
\label{ThmLeray}
For any $u_0 \in L^{2} ({\mathcal{F}_0};\R^{2})$ and $(\ell_0 , r_0) \in \R^2 \times \R$ satisfying  \eqref{CondCompatibilite},
for any  $T>0$,  there exists a weak solution 
\begin{equation*}
(\ell,r,\overline{u}) \in  C^0 ([0,T]; \R^2 \times \R) \times [C ([0,T]; L^2 (\Omega)) \cap L^2 (0,T; H^{1}(\Omega))],
\end{equation*}
of \eqref{NS1}-\eqref{Solide2NS} with the initial data \eqref{Solideci}-\eqref{Eulerci2}.
Moreover, for any $t\in [0,T]$, 
\begin{equation} \label{EE}
\frac{1}{2}\int_{\Omega } \rho(t,\cdot)  |\overline{u} (t,\cdot)|^{2} \, dx
+ 2 \int_{(0,t) \times \Omega } \rho(s,x) \, D\overline{u}(s,x):D\overline{u}(s,x) \, dx \, ds
= \frac{1}{2} \int_{\Omega }  \rho_{0}(x) |\overline{u}_{0}(x)|^2 \, dx.
\end{equation}
\end{Theorem}
\begin{Remark}
The proof of the existence of such weak solutions can be found in \cite{hs,SMST}. Note in particular that the function $\rho_{{\mathcal S}}$ defined by $\rho_{{\mathcal S}}(t,x) = \rho_0 ( (\tau^{\ell,r} (t,\cdot))^{-1}(x) )$ in ${\mathcal S}^{\ell,r}(t)$ and $\rho_{{\mathcal S}}(t,x) = 0$ in  ${\mathcal F}^{\ell,r}(t)$ is a weak solution of
\begin{equation} \label{Eq:TransportRho}
\left\{ \begin{array}{l}
 \partial_{t} \rho_{{\mathcal S}} + \div (\rho_{{\mathcal S}} \overline{u}) =0 \ \text{ in } (0,T) \times \Omega, \\
 \rho_{{\mathcal S}}(0,\cdot) = \rho_{{\mathcal S}_{0}} \text{ in } {\mathcal S}_{0} \ \text{ and } \ \rho_{{\mathcal S}}(0,\cdot) = 0 \text{ in } {\mathcal F}_{0},
\end{array} \right.
\end{equation}
and hence the unique solution of this system (see \cite[Corollary II.1]{DiPernaLions}). \par
We notice that the notion of weak solutions can be slightly different. In particular, \cite{DE-CPDE} does not express the solid movement by \eqref{Compatibilite} or $\rho$ by \eqref{Eq:Rho}, but as follows. The solid density $\rho_{{\mathcal S}}$ is obtained as the solution of \eqref{Eq:TransportRho} and then the compatibility condition in \cite{DE-CPDE} reads: $\overline{u}(t,\cdot)$ belongs to $H^{1}(\Omega)$ for almost every $t$ and
\begin{equation*}
\rho_{{\mathcal S}} D \overline{u} =0.
\end{equation*}
Due to the lack of regularity of $\overline{u}$, we do not know if this compatibility condition is sufficient to ensure \eqref{Compatibilite} (see also the discussion in \cite[Section 3]{FE3}). \par
\end{Remark}
\begin{Remark}
The energy identity \eqref{EE} belongs to the  folklore in the subject and can be proved proceeding as in the case of a fluid alone, see for instance \cite[p. 87]{lions}.
The strong continuity in time of $\overline{u}$ in $L^{2}(\Omega)$ is then a direct consequence of \eqref{EE}.
\end{Remark}
Let us add a few words on previous references. In the case when $\Omega= \R^{2}$, one can again use a rigid change of variables to prove the existence and uniqueness of such solutions cf.  \cite{Judakov,Serre,TakTuc}. In the case considered here where $\Omega$ is bounded, this is no longer possible; we refer here to \cite{GLS,DE-CPDE,Conca1,SMST} which establish the existence of  solutions ``\`a la Leray'' as stated in Theorem \ref{ThmLeray}.
Let us also mention the recent works  \cite{FE1,FE2} which establish the existence of  solutions ``\`a la Leray''  in three dimensions and the papers \cite{GM,DE2,Tak} where the existence and uniqueness of strong solutions for short times were studied, including in the three-dimensional case. \par
\ \par
The second main result of this paper states  that the solution given by Theorem \ref{ThmLeray} is unique as long as there is no collision. 
\begin{Theorem} \label{UniqLeray}
Let $T>0$ and $(\ell,r,u)$ be as in Theorem \ref{ThmLeray}. 
Assume that for any $t \in [0,T]$, $\dist ({\mathcal S}(t), \partial \Omega) ) > 0$.
Let $(\tilde{\ell},\tilde{r},\tilde{u}) $ be another weak solution of \eqref{NS1}-\eqref{Solide2NS} on $[0,T]$ with the same initial data. 
Then $(\tilde{\ell},\tilde{r},\tilde{u}) = (\ell,r,u)$.
\end{Theorem}
This result extends the one in  \cite{Tak} where it was assumed in addition that the initial fluid velocity is  in the Sobolev space $H^1$.
Therein it was mentioned that 
``uniqueness of weak solutions is an open question, even in the two-dimensional case.''
This issue  was also mentioned recently in the conclusion of the paper \cite{DR}.
Theorem \ref{UniqLeray} therefore brings an answer to this issue, as long as there is no collision. \par
It is not known in general whether or not a collision may happen.
However the possibility of a collision is excluded in some particular cases by the results in \cite{Hillairet,Hesla}, see also the recent work  \cite{DGVH} about the influence of the boundary regularity. 
On the other hand the results of  \cite{hs,Staro} prove that such weak solutions cannot be unique if a collision occurs.
\subsection{Structure of the paper}
To simplify the notations and without loss of generality, we will suppose that 
\begin{equation*}
h_{0}=0.
\end{equation*}
The paper is organized as follows. In Section \ref{BasicL} we establish a preliminary result on a class of changes of variables associated to a rigid  motion. This result will be useful for the proof of Theorem  \ref{UniqYudo} and  for the one of Theorem  \ref{UniqLeray} as well.
Then  we proceed with the proofs of these theorems respectively in Section  \ref{Section:Yudo} and in Section \ref{Section:Leray}.
The structure of these sections is quite the same: we will start by giving some a-priori estimates satisfied by any solution, respectively in the Sections \ref{Section:PrioriYudo} and \ref{Section:PrioriLeray}, and then we prove the uniqueness, in the Sections \ref{section:UniqYudo} and \ref{section:UniqLeray}.
Finally, in the appendix, we prove Theorem \ref{ThmYudo}. \par
\section{A basic lemma}
\label{BasicL}
Given $A \subset \R^{2}$ and $\delta>0$, we denote
\begin{equation*}
{\mathcal V}_{\delta}(A) := \Big\{ x \in \R^{2} \ \Big/ \ \dist (x,A) \leq \delta \Big\}.
\end{equation*}
We rely on the following proposition. \par
%
%\ \par
%
\begin{Proposition} \label{ProDiffeos}
Let $\Omega$ and ${\mathcal S}_{0}$ be fixed as previously.
There exist a compact neighborhood $U$ of $\Id$ in $SE(2)$, $\delta>0$ and $\Psi \in C^{\infty}(U;\mbox{Diff}(\overline{\Omega}))$ such that $\Psi[\Id]=\Id$ and that for all $\tau \in U$,
\begin{gather} 
\label{PsiVolumePreserving}
\Psi[\tau] \text{ is volume-preserving}, \\
\label{EqPsiTau}
	\Psi[\tau] (x) = \tau(x)  \text{ on } {\mathcal V}_{\delta}({\mathcal S}_{0}) \ \text{ and } \ 
	\Psi[\tau] (x) = x  \text{ on } {\mathcal V}_{\delta}(\partial \Omega) \cap \overline{\Omega}.
\end{gather}
\end{Proposition}
Above, $\mbox{Diff}(\overline{\Omega})$ denotes the set of $C^{\infty}$-diffeomorphisms of $\overline{\Omega}$. 
\begin{proof}
The proof is similar to \cite[Lemma 1]{GSgeodesic}. First we use that the exponential map $\exp : se(2) \rightarrow SE(2)$ is locally a diffeomorphism near the origin of $se(2)$, say on a neighborhood ${\mathcal U} \subset SE(2)$ of $\Id_{\R^{2}}$.  Here $se(2)$ is the Lie algebra associated to the Lie group $SE(2)$. This exponential map on the space $se(2) \sim \R^{2} \times \R$ can be represented as the map which associates to $(\ell,r) \in \R^{2} \times \R$ the value at time $1$ of the solution $\tau(t)$ of following ODE in $SE(2)$:
\begin{equation} \label{EDOtau}
\frac{d}{dt} \tau(t, x)  = \ell + r \, (\tau(t, x) - \tau(t, 0))^{\perp}  = \ell + r \, (\tau(t, x) - t \ell)^{\perp} \ \text{ with } \tau(0,\cdot)=\Id_{\R^{2}}.
\end{equation}
Reducing ${\mathcal U}$ if necessary, we find $U$ as a compact neighborhood of $\Id$ in $SE(2)$ on which $\ln$ is a diffeomorphism, which contains all the intermediary states $\tau(t,\cdot)$ leading to $\tau(1)=\tau$ when $\tau \in U$ and for which holds for some $\delta>0$
\begin{equation*}
\max \Big\{ |\tau(x)- x|, \ x \in {\mathcal S}_{0}, \ \tau \in U \Big\} \leq \delta \ \text{ and } \ 
\max \Big\{ \dist (\tau({\mathcal S}_{0}), \partial \Omega) , \ \tau \in U \Big\} \geq 3 \delta.
\end{equation*}
Now given $\tau \in U$, we hence associate $(\ell,r):= \ln(\tau)$ and the corresponding time-dependent $\tau(t,x)$.
Let $\phi(t,x)$ a smooth function equal, for each $t \in [0,1]$, to $1$ in ${\mathcal V}_{\delta}(\tau(t,{\mathcal S}_{0}))$ and to $0$ outside of ${\mathcal V}_{2\delta}(\tau(t,{\mathcal S}_{0}))$.
We define the following time-dependent vector field on $\R^{2}$:
\begin{equation*}
V_{\tau}(t,x) := \nabla^{\perp} \left( \phi(x) \left( x^{\perp} \cdot \ell + \frac{|x - t \ell|^{2}}{2} r \right) \right).
\end{equation*}
Note that
\begin{equation} \label{ValV}
V_{\tau}(t,x) = \ell + r \, (x-t \ell)^{\perp} \text{ in } {\mathcal V}_{\delta}({\mathcal S}_{0}) \ \text{ and } \ 
V_{\tau}(t,x) = 0 \text{ in } {\mathcal V}_{\delta}(\partial \Omega).
\end{equation}
We define $\Psi \in \mbox{Diff}(\overline{\Omega})$ as the value at $t=1$ of the flow associated to $V$, that is
\begin{equation*}
\Psi[\tau] := \gamma(1,\cdot),
\end{equation*}
where $\gamma(t,x)$ the solution of the ODE:
\begin{equation*}
\frac{d}{dt} \gamma(t,x)  = V(t,\gamma(t,x)) \ \text{ with } \ \gamma(0,\cdot)=\Id_{\overline{\Omega}}.
\end{equation*}
It is straightforward to see that $\gamma$ is a smooth function of $V$ and hence that $\Psi$ is a smooth function of $\tau$. Also, \eqref{EqPsiTau} follows from \eqref{ValV}. Finally \eqref{PsiVolumePreserving} follows from $\div V=0$.
\end{proof}
We have the next corollary, where we consider $SE(2) \subset \R^{3}$ so that we can use the $\R^{3}$ norm on the elements of $SE(2)$. When we consider a time-dependent family of rigid motions $(\tau(t))_{t \in [0,T]}$, we will write $\tau_{t}:=\tau(t,\cdot)$.
\begin{Corollary} \label{CorEstPsi}
Reducing $U$ if necessary one has for some $C>0$:,
\begin{equation} \label{EqDiffTau1}
\forall \tau, \tilde{\tau} \in U, \ \ \| \Psi[\tau] \circ \{ \Psi[\tilde{\tau}] \}^{-1} - \Id \|_{C^{2}(\overline{\Omega})} \leq C \| \tau - \tilde{\tau} \|_{\R^{3}},
\end{equation}
and if $\tau_{t}, \tilde{\tau}_{t} \in C^{1}([0,T];SE(2))$, then for all $t_{0} \in [0,T]$,

\begin{equation} \label{EqDiffTau2}
\left\| \left[ \frac{d}{dt} \left( \Psi[\tau_{t}] \circ \{ \Psi[\tilde{\tau}_{t}] \}^{-1} \right) \right]_{t=t_{0}} \right\|_{C^{1}(\overline{\Omega})} 
\leq C \Big(  \| \tilde{\tau}'_{t_{0}} \|_{\R^{3}} \, \| \tau_{t_{0}} - \tilde{\tau}_{t_{0}} \|_{\R^{3}}
+ \| \tau'_{t_{0}} - \tilde{\tau}'_{t_{0}} \|_{\R^{3}} \Big).
\end{equation}
\end{Corollary}
Let us emphasize that $\{ \Psi[\tilde{\tau}_{t}] \}^{-1}$ denotes the inverse of $\Psi[\tilde{\tau}_{t}]$ with respect to the variable $x$. 
\begin{proof}
Reducing $U$ if necessary, one has uniformly for $\tau \in U$ that
\begin{equation*}
\| \Psi[\tau] - \Id \|_{C^{2}(\overline{\Omega})} \leq \frac{1}{2},
\end{equation*}
so that we have a uniform bound on $\| \{ \Psi[\tau] \}^{-1} \|_{C^{2}(\overline{\Omega})}$, and \eqref{EqDiffTau1} follows from the fact that $\Psi$ is uniformly Lipschitz on $U$. In the same way, we have
\begin{equation} \label{EqDiffTau1bis}
\forall \tau, \tilde{\tau} \in U, \ \ 
\| \partial_{x} \Psi[\tau]  - \partial_{x} \Psi[\tilde{\tau}] \|_{C^{1}(\overline{\Omega}; \R^{2 \times 2})} \leq C \| \tau - \tilde{\tau} \|_{\R^{3}}.
\end{equation}
On the other side, denoting 
\begin{equation*}
g(t,x):=\Psi[\tau_{t}](x),\ \ 
h(t,x):=\{ \Psi[\tau_{t}] \}^{-1}(x), \ \ 
\tilde{g}(t,x):=\Psi[\tilde{\tau}_{t}](x) \ \text{ and } \ 
\tilde{h}(t,x):=\{ \Psi[\tilde{\tau}_{t}] \}^{-1}(x),
\end{equation*}
we have
\begin{equation*}
\frac{d}{dt} \left( \Psi[\tau_{t}] \circ \{ \Psi[\tilde{\tau}_{t}] \}^{-1} (x)\right) =
\partial_{t} g (t,\tilde{h}(t,x)) + \big( \partial_{x} g (t,\tilde{h}(t,x)) \big) \partial_{t}\tilde{h}(t,x).
\end{equation*}
Since
\begin{equation} \label{rel0}
\partial_{t} \tilde{g} (t,\tilde{h}(t,x)) + \big( \partial_{x} \tilde{g} (t,\tilde{h}(t,x)) \big) \partial_{t}\tilde{h}(t,x) = 0,
\end{equation}
we have
\begin{eqnarray}
\nonumber
\frac{d}{dt} \left( \Psi[\tau_{t}] \circ \{ \Psi[\tilde{\tau}_{t}] \}^{-1} (x) \right) &=&
\partial_{t} g (t,\tilde{h}(t,x)) - \partial_{t} \tilde{g} (t,\tilde{h}(t,x)) \\
\label{diffpsi}
& \, & \! + \  \big\{ \partial_{x} g (t,\tilde{h}(t,x)))
- \partial_{x} \tilde{g} (t,\tilde{h}(t,x)) \big\} \, \partial_{t}\tilde{h}(t,x).
\end{eqnarray}
Concerning the first term in the right hand side of \eqref{diffpsi}, we use
\begin{equation} \label{partgt}
\partial_{t} g (t,y) =[d \Psi (\tau_{t}) \cdot \tau'_{t}] (y)
 \ \text{ and } \
\partial_{t} \tilde{g} (t,y)  = [d \Psi (\tilde{\tau}_{t}) \cdot \tilde{\tau}'_{t}] (y),
\end{equation}
and the regularity of $\Psi$. Concerning the second one, we use \eqref{EqDiffTau1bis} to estimate the term between brackets and \eqref{rel0} and \eqref{partgt} to estimate $\partial_{t}\tilde{h}$. Our claim \eqref{EqDiffTau2} follows.
\end{proof}
\begin{Remark}
\label{nems}
Clearly we could have put any $C^{k}(\overline{\Omega})$ norm on the left hand sides of \eqref{EqDiffTau1} and \eqref{EqDiffTau2}.
\end{Remark}
\section{Proof of Theorem \ref{UniqYudo}}
\label{Section:Yudo}
In this section, we consider the inviscid case and prove Theorem \ref{UniqYudo}. \par
\subsection{A priori estimates}
\label{Section:PrioriYudo}
We begin by giving a priori estimates on a solution given by Theorem \ref{ThmYudo}. We suppose that $\partial \Omega$ has $g+1$ connected components $\Gamma_{1}, \dots , \Gamma_{g+1}$; we suppose that $\Gamma_{g+1}$ is the outer one. We add to this list $\Gamma_{0}=\Gamma_{0}(t)=\partial {\mathcal S}(t)$. We denote $\mathfrak{t}$ the tangent to $\partial \Omega$ and $\partial {\mathcal S}(t)$ and define
\begin{equation*}
\gamma_{0}^{i} := \int_{\Gamma_{i}} u_{0} \cdot \mathfrak{t} \, d \sigma \ \text{ for } i=1, \dots, g \ \text{ and } \ 
\gamma_{0} := \int_{\partial {\mathcal S}_{0}} u_{0} \cdot \mathfrak{t} \, d \sigma,
\end{equation*}
and we let 
\begin{equation*}
\gamma := |\gamma_{0}| + \sum_{i=1}^{g}|\gamma_{0}^{i}|.
\end{equation*}
We have the following a priori estimates on any solution of the system in the sense of Theorem \ref{ThmYudo}.
\begin{Proposition} \label{PropAPrioriEstimates}
Let $(\ell,r,u)$ a solution of the system in the sense of Theorem \ref{ThmYudo} in the time interval $[0,T]$. Then one has the following a priori estimates: for all $t \in [0,T]$ and $q \in [1,+\infty]$,
\begin{gather*}
\| \curl u(t,\cdot) \|_{L^{q}({\mathcal F}(t))} = \| \curl u_{0} \|_{L^{q}({\mathcal F}_{0})}, \\
\forall i = 1, \dots , g, \ 
\int_{\Gamma_{i}} u(t,\cdot) \cdot \mathfrak{t} \, d \sigma = \gamma^{i}_{0} \ \text{ and  } \
\int_{\partial {\mathcal S}(t)} u(t,\cdot) \cdot \mathfrak{t} \, d \sigma = \gamma_{0}, \\
\| u(t,\cdot) \|^{2}_{L^{2}({\mathcal F}(t))} + m | \ell (t)|^{2} + {\mathcal J} |r(t)|^{2}
= \| u_{0} \|^{2}_{L^{2}({\mathcal F}_{0})} + m | \ell_{0}|^{2} + {\mathcal J}|r_{0}|^{2}.
\end{gather*}
Moreover, for $\delta>0$, there is a constant $C>0$ such that for all $T$ such that $\dist ({\mathcal S}(t),\partial \Omega) \geq \delta$ on $[0,T]$, one has for all $t \in [0,T]$ and $q \in [2,\infty)$,
\begin{equation} \label{EstYudo}
\| u(t,\cdot) \|_{W^{1,q}({\mathcal F}(t))} \leq Cq \big( \| \omega_{0} \|_{L^{q}({\mathcal F}_{0})}
+ |\ell_{0}| + |r_{0}| + \gamma\big).
\end{equation}
\end{Proposition}
\begin{proof}
Given such a solution $(\ell,r,u)$, the vorticity $\omega(t,x):=\curl u(t,x)$ satisfies the transport equation
\begin{equation} \label{EqVor}
\partial_{t} \omega + (u \cdot \nabla) \omega =0 \ \text{ in } \ {\mathcal F}(t).
\end{equation}
Due to the log-Lipschitz regularity of $u$, one can associate a unique flow $\Phi=\Phi(t,s,x)$, and by uniqueness of the solutions of \eqref{EqVor} at this level of regularity, one has
$\omega(t,x) = \omega_{0}(\Phi(0,t,x))$. Since $\Phi$ is volume-preserving (as follows from $\div u =0$), we obtain the claim on $\| \curl u \|_{L^{q}({\mathcal F}(t))}$. The second conservation is Kelvin's theorem, and the third one the conservation of energy. \par
Estimate \eqref{EstYudo} is classical in the case of a fluid alone, and is central in the argument of Yudovich \cite{Yudovich}. Here, we only need to prove that the constant appearing in the elliptic estimate for the $\div$/$\curl$ system does not depend on the position of the solid, as long as it stays distant from the boundary. Precisely, we prove the following.
\begin{Lemma} \label{LemEstLp}
For any $R>0$, there exists $C>0$ such that if ${\mathcal S} = \tau ({\mathcal S}_{0})$ for $\tau \in SE(2)$ satisfies
\begin{equation} \label{ContrainteS}
{\mathcal S} \subset \Omega \ \text{ and } \ \dist ( {\mathcal S}, \partial \Omega) \geq R,
\end{equation}
then any $u : {{\mathcal F}} \rightarrow \R^{2}$ verifies, for all $q \geq 2$:
\begin{equation} \label{EllipticEstimate}
\| u \|_{W^{1,q}({\mathcal F})} \leq C q \Big( \| \curl u \|_{L^{q}({\mathcal F})} + \| \div u \|_{L^{q}({\mathcal F})} \Big) 
+ C \Big( \| u \cdot n \|_{W^{1-1/q,q}(\partial  {\mathcal F})} + \sum_{i=0}^{g} \left| \int_{\Gamma_{i}}  u \cdot \mathfrak{t} \, d \sigma \right| \Big),
\end{equation}
where $\Gamma_{0}:= \partial{\mathcal S}$ and ${\mathcal F} := \Omega  \setminus {\mathcal S}$.
\end{Lemma}
Above we take as a convention that
\begin{equation} \label{ConventionSF}
\| f \|_{W^{1-1/q,q}(\partial {\mathcal F})} := \inf \big\{ \| \overline{f} \|_{W^{1,q}({\mathcal F})}, \ \overline{f} \in W^{1,q}( {\mathcal F}) \ \text{ and } \ \overline{f}_{|\partial {\mathcal F}} = f \big\}.
\end{equation}
That this norm is equivalent to the usual one (for fixed $q$), comes from the trace theorem and the existence of a continuous extension operator $W^{1-1/q,q}( \partial {\mathcal F}) \rightarrow W^{1,q}( {\mathcal F})$. 
\par
\ \par
Once Lemma \ref{LemEstLp} is established, \eqref{EstYudo} is a consequence of the previous conservations.
\end{proof}
Note that the equivalent of Lemma \ref{LemEstLp} in the framework of H\"older spaces is known (see e.g. \cite[Lemma 5]{ogfstt}):
\begin{Lemma} \label{LemHolder}
In the context of Lemma \ref{LemEstLp}, for $\lambda \in \N$ and $\alpha \in (0,1)$, there exists a constant $C>0$ independent of $\tau$ such that
\begin{equation}
\label{reg1}
\| u \|_{C^{\lambda+1,\alpha}({\mathcal F})} \leq C \left( \|\div  u\|_{C^{\lambda,\alpha}({\mathcal F})} 
+ \| \curl  u \|_{C^{\lambda,\alpha}({\mathcal F})} + \| u\cdot n\|_{C^{\lambda+1,\alpha}(\partial \Omega \cup \partial {\mathcal S})}
+ \sum_{i=0}^{g} \left| \int_{\Gamma_{i}}  u \cdot \mathfrak{t} \, d \sigma \right|  \right).
\end{equation}
\end{Lemma}
\noindent
\begin{proof}[Proof of Lemma \ref{LemEstLp}] 
As we explained this is standard in a fixed domain (see in particular \cite{GT,Yudovich}). Note in particular that it is elementary to reduce to the case where 
\begin{equation*}
u \cdot n =0 \ \text{ on } \ \partial \Omega \cup  \partial {\mathcal S} \ \text{ and } \ \int_{\Gamma_{i}}  u \cdot \mathfrak{t} \, d \sigma =0 \ \text{ for all } \ i= 0 \dots g,
\end{equation*}
by using the convention \eqref{ConventionSF} and the following functions $H_{i}:= \nabla^{\perp} \psi_{i}$, for $1 \leq i \leq g$, where $\psi_{i}$ satisfies
\begin{equation} \nonumber
\left\{ \begin{array}{l}
-\Delta {\psi}_i = 0 \ \text{ for } \ x \in \mathcal{F} , \\
\psi_{i} = 1 \ \text{ on } \ \Gamma_{i}, \\
\psi_{i} = 0 \ \text{ on } \ (\partial \Omega \cup  \partial {\mathcal S} ) \setminus \Gamma_{i}, 
\end{array} \right.
\end{equation}
so that 
\begin{equation*}
\curl H_{i} = \div H_{i}=0 \ \text{ in } \ \Omega \setminus {\mathcal F}, \ \
H_{i} \cdot n =0 \ \text{ on } \ \partial \Omega \cup \partial {\mathcal S} \ \text{ and } \ \int_{\Gamma_{j}} H_{i} \cdot \mathfrak{t} \, d \sigma = \delta_{ij} \ \text{ for } \ 1 \leq j \leq g.
\end{equation*}
For these functions $H_{i}$, we have suitable estimates by using Lemma \ref{LemHolder}. \par
\ \par
We notice that the set of $\tau \in SE(2)$ such that ${\mathcal S} = \tau({\mathcal S}_{0})$ satisfies \eqref{ContrainteS} is compact. Hence by a straightforward compactness argument, and since such a constant $C>0$ is well-defined for any fixed configuration $\hat{\mathcal S} = \hat{\tau}({\mathcal S}_{0})$ satisfying \eqref{ContrainteS}, we see that we only need to prove that, given such a fixed configuration $\hat{\mathcal S}$, there exists a constant $C>0$ for which \eqref{EllipticEstimate} is valid whenever ${\mathcal S}= \tau (\hat{\mathcal S})$, when $\tau$ belongs to some arbitrarily small neighborhood of $\Id_{\R^{2}}$. Now given a fixed configuration $\hat{\mathcal S}$, we introduce  $\delta >0$ and $\varepsilon \in (0,\delta)$ such that for any $\tau \in SE(2)$, $\| \tau - \Id \| \leq \varepsilon$ one has for ${\mathcal S}:= \tau(\hat{\mathcal S})$:
\begin{gather*}
\dist ({\mathcal S}, \partial \Omega) \geq 4 \delta \ \text{ and } \
{\mathcal V}_{\delta}(\partial \hat{\mathcal S}) \text{ is a tubular neighborhood of } \partial \hat{\mathcal S}.
\end{gather*}
We introduce $\varphi$ a cutoff function in $C^{\infty}_{0}(\R^{2})$ such that
\begin{equation*}
\varphi =1 \text{ on } {\mathcal V}_{2\delta}(\hat{\mathcal S}) \ \text { and } \
\varphi =0 \text{ on } \R^{2} \setminus {\mathcal V}_{3\delta}(\hat{\mathcal S}).
\end{equation*}
Now we introduce $u_{1} \in W^{1,q}_{loc}(\R^{2} \setminus {\mathcal S})$ and $u_{2} \in W^{1,q}({\Omega})$ as the solutions of the following elliptic systems:
\begin{equation} \label{Eq:u1u2}
\left\{ \begin{array}{l}
 \curl u_{1} = \varphi \curl u \ \text{ in } \ \R^{2} \setminus {\mathcal S}, \\
 \div u_{1} = \varphi \div u \ \text{ in } \ \R^{2} \setminus {\mathcal S}, \\
 u_{1} \cdot n = 0 \ \text{ on } \ \partial {\mathcal S}, \\
 \int_{\partial {\mathcal S}} u_{1} \cdot \mathfrak{t} \, d \sigma= 0, \\
 \lim_{|x| \rightarrow +\infty} u_{1}(x)=0,
\end{array} \right.
\ \qquad
\left\{ \begin{array}{l}
 \curl u_{2} = (1-\varphi) \curl u \ \text{ in } \ \Omega, \\
 \div u_{2} = (1-\varphi) \div u \ \text{ in } \ \Omega, \\
 u_{2} \cdot n = - u_{1} \cdot n\ \text{ on } \ \partial \Omega, \\
 \int_{\partial \Gamma_{i}} u_{2} \cdot \mathfrak{t} \, d \sigma=0, \ i=1 \dots g.
\end{array} \right.
\end{equation}
Note that the compatibility condition 
\begin{equation*}
- \int_{\partial \Omega} u_{1} \cdot n \, d \sigma = \int_{\Omega}(1-\varphi) \div u \, dx,
\end{equation*}
comes from 
\begin{equation*}
\int_{\Omega} (1-\varphi) \div u \, dx = \int_{{\mathcal F}} (1-\varphi) \div u \, dx,
\end{equation*}
and the use of the divergence theorem:
\begin{equation*}
\int_{{\mathcal F}} \div u \, dx = 0 \ \text{ and } \
\int_{{\mathcal F}} \varphi \div u \, dx = \int_{{\mathcal F}} \div u_{1} \, dx = 
\int_{\partial \Omega} u_{1} \cdot n \, d \sigma + \int_{\partial {\mathcal S}} u_{1} \cdot n \, d \sigma.
\end{equation*}
Moreover, by using Stokes' formula (in the interior of $\Gamma_{i}$ of in ${\mathcal S}$), the fact that the outer component of $\partial \Omega$ is $\Gamma_{g+1}$ and considering the supports of $\varphi$ and $1-\varphi$, one sees that
\begin{equation*}
\forall i= 1, \dots, g, \ \ \int_{\Gamma_{i}} u_{1} \cdot \mathfrak{t} \, d \sigma = \int_{\partial {\mathcal S}} u_{2} \cdot \mathfrak{t} \, d \sigma =0.
\end{equation*}
Using the fact that inequality \eqref{EllipticEstimate} is true for a fixed geometry and that the problem satisfied by $u_{1}$ is invariant under rigid movements, we deduce that there exists a constant $C >0$ independent of $\tau$ (satisfying $\| \tau - \Id \| \leq \varepsilon$) and for which
\begin{equation} \label{EstU1}
\| u_{1} \|_{W^{1,q}(\R^{2} \setminus {\mathcal S})} 
\leq C q \big( \| \curl u \|_{L^{q}({\mathcal F})} + \| \div u \|_{L^{q}({\mathcal F})} \big) . 
\end{equation}
and
\begin{equation} \label{EstU2}
 \| u_{2} \|_{W^{1,q}(\Omega)}
\leq C q \big( \| \curl u \|_{L^{q}({\mathcal F})} + \| \div u \|_{L^{q}({\mathcal F})}  \big)   + C_{1}  \| u_{1} \cdot n \|_{W^{1-1/q,q}(\partial \Omega)} . 
\end{equation}
As a consequence, the right hand sides of \eqref{EstU1} and \eqref{EstU2} can be estimated by the right hand side of \eqref{EllipticEstimate}. \par
\ \par
Now we introduce $w : {\mathcal F} \rightarrow \R$ as the solution of
\begin{equation*}
\left\{ \begin{array}{l}
 \curl w = \div w =0 \ \text{ in } \ {\mathcal F}, \\
 w \cdot n = - u_{2} \cdot n \ \text{ on } \ \partial {\mathcal S}, \\
 w \cdot n = 0 \ \text{ on } \ \partial \Omega, \\
\int_{\Gamma_{i}} w \cdot \mathfrak{t} \, d \sigma=  0 , \ i=0,\dots,g.
\end{array} \right.
\end{equation*}
Note that the compatibility condition between $\div w$ and $w \cdot n$ is satisfied because, relying on the support of $\varphi$, one has
\begin{equation*}
\int_{\partial {\mathcal S}} u_{2} \cdot n \, d \sigma = \int_{ {\mathcal S}} (1- \varphi) \div u \, dx =0.
\end{equation*}
We observe that $u_{2}$ is harmonic in ${\mathcal V}_{\delta}(\partial {\mathcal S})$. It follows from standard properties of harmonic functions that for some $C >0$ independent of $\tau$ small and $q \geq 2$ one has (given $\alpha \in (0,1)$):
\begin{equation*}
\| u_{2|\partial {\mathcal S}} \|_{C^{1,\alpha}(\partial {\mathcal S})} \leq C \| u_{2} \|_{L^{2}({\mathcal V}_{\delta}(\partial {\mathcal S}))}.
\end{equation*}
It follows that  $w \in C^{1,\alpha}({\mathcal F})$ and using Lemma \ref{LemHolder} we deduce that for some $C, C', C''>0$ independent of $\tau$ small:
\begin{equation*}
\| w \|_{W^{1,q}({\mathcal F})} \leq C \| w \|_{C^{1,\alpha}({\mathcal F})} 
\leq C' \| u_{2} \|_{L^{2}({\mathcal F})}
\leq C'' \| u_{2} \|_{W^{1,q}({\mathcal F})} .
\end{equation*}
The conclusion follows since by uniqueness of the solutions of the $\div$/$\curl$ system:
\begin{equation*}
\left\{ \begin{array}{l}
 \curl v = 0 \ \text{ in } \ {\mathcal F}, \\
 \div v = 0 \ \text{ in } \ {\mathcal F}, \\
 v \cdot n = 0 \ \text{ on } \ \partial \Omega \cup \partial {\mathcal S}, \\
 \int_{\Gamma_{i}} v \cdot \mathfrak{t} \, d \sigma= 0, \ i=0, \dots, g,
\end{array} \right. \ \Longrightarrow \ v =0,
\end{equation*}
one has:
\begin{equation*}
u = u_{1} + u_{2} + w.
\end{equation*}
Gathering the estimates above, we get the conclusion.
\end{proof}
We have the following consequence of Proposition \ref{PropAPrioriEstimates}.
\begin{Corollary} \label{CorAPrioriEstimates}
Under the assumptions of Proposition \ref{PropAPrioriEstimates} (including that $\dist ({\mathcal S}(t),\partial \Omega) \geq \delta$ on $[0,T]$), we have for some constant $C=C ( \| \omega_{0} \|_{L^{\infty}({\mathcal F}_{0})} + | \ell_{0} | + |r_{0}| + \gamma )>0$ that uniformly in $[0,T]$:
\begin{equation} \label{DeuxiemeVagueAPE}
\| {u}(t) \|_{H^{1}({\mathcal F}(t))}
+ \| \partial_{t} u \|_{L^{2}({\mathcal F}(t))}
+ \| \nabla p \|_{L^{2}({\mathcal F}(t))}
\leq C.
\end{equation}
\end{Corollary}
\begin{proof}[Proof of Corollary \ref{CorAPrioriEstimates}]
The estimate of $\| {u}(t) \|_{H^{1}({\mathcal F}(t))}$ is a direct consequence of Proposition \ref{PropAPrioriEstimates}.
Also, by Proposition \ref{PropAPrioriEstimates}, we have that 
\begin{equation} \label{BorneW13}
\| u \|_{W^{1,4}(\overline{\mathcal F}(t))}
\leq C \big( \| \omega_{0} \|_{L^{\infty}({\mathcal F}_{0})} + | \ell_{0} | + |r_{0}| + \gamma \big).
\end{equation}
Now we use the decomposition of $\nabla p$ (see e.g. \cite[Lemma 3]{ogfstt}):
\begin{equation} \label{DecompP}
\nabla p = \nabla \mu - \nabla \left( (\Phi_{i})_{i=1,2,3} \cdot \begin{bmatrix}  \ell \\ r \end{bmatrix}' \right),
\end{equation}
where the functions $\Phi_{i}=\Phi_{i}(t,x)$ (known as the Kirchhoff potentials) and the function $\mu=\mu(t,x)$ are the solutions of the following problems: 
\begin{equation}  \label{KirchoffPotentials}
\left\{ \begin{array}{l}
-\Delta {\Phi}_i = 0 \ \text{ for } \ x \in \mathcal{F}(t), \\
\displaystyle \frac{\partial {\Phi}_i}{\partial n}= K_i  \ \text{ for } \  x \in \partial \mathcal{S}(t), \smallskip \\
\displaystyle \frac{\partial {\Phi}_i}{\partial n}= 0  \ \text{ for } \  x \in \partial \Omega,
\end{array} \right.
\quad \text{ where } \quad
K_i := \left\{\begin{array}{ll} 
n_i & \text{if} \  i=1,2 ,\\ 
(x-h(t))^\perp \cdot n & \text{ if } \ i=3,
\end{array}\right.
\end{equation}
and
\begin{equation*}
\left\{ \begin{array}{l}
-\Delta \mu = \tr (\nabla u \cdot \nabla u) \ \text{ for } \ x \in \mathcal{F}(t), \\
\frac{\partial \mu}{\partial n}= \nabla^{2} \rho \, \{ u- u_\mathcal{S} , u- u_\mathcal{S} \} - n \cdot \big(r \left(2u-u_\mathcal{S}-\ell \right)^{\perp}\big)  \ \text{ for } \ x \in \partial \mathcal{S}(t), \\
\frac{\partial \mu}{\partial n}= -\nabla^{2} \rho (u,u)  \ \text{ for } \ x \in \partial \Omega, \\
\end{array} \right.
\end{equation*}
where $u_\mathcal{S}=u_\mathcal{S}(t,x)$ is given by \eqref{vietendue} and where $\rho=\rho(t,x)$ is the signed distance to $\partial \Omega \cup \partial {\mathcal S}(t)$ (which we define in a neighborhood of $\partial \Omega \cup \partial {\mathcal S}(t)$), chosen to be negative inside ${\mathcal F}$. The function $\rho$ is constant in time near $\partial \Omega$, and is transported by the solid movement near $\partial {\mathcal S}(t)$.
Note that the fact the compatibility condition between $\Delta \mu$ and $\frac{\partial}{\partial n} \mu$ is satisfied thanks to 
\begin{equation*}
\tr(\nabla u \cdot \nabla u) = \div ((u \cdot \nabla) u), \ \ \nabla \rho = n \ \text{ on } \ \partial \Omega \cup \partial {\mathcal S},
\end{equation*}
\eqref{Euler3} and \eqref{Euler4}. \par
Moreover, using  Green's theorem, \eqref{KirchoffPotentials} and \eqref{DecompP},
we obtain that the equations for the solid, that is \eqref{Solide1}-\eqref{Solide2}, can be recast as follows (see also \cite[Lemma 4]{ogfstt}):
\begin{equation} \label{EvoMatrice}
\mathcal{M} \begin{bmatrix} \ell \\ r \end{bmatrix}' 
= \begin{bmatrix}  \displaystyle\int_{ \mathcal{F}(t)} \nabla \mu \cdot \nabla \Phi_i \, dx   \end{bmatrix}_{i \in \{1,2,3\}} ,
\end{equation}
where
\begin{equation}
\label{Added}
\mathcal{M} :=\mathcal{M}_1+\mathcal{M}_2,
\quad
\mathcal{M}_1 := \begin{bmatrix} m \Id_2 & 0 \\ 0 & \mathcal{J}\end{bmatrix}
\quad \text{ and } \quad
\mathcal{M}_2: = \begin{bmatrix} \int_{\mathcal{F} (t)} \nabla \Phi_i \cdot \nabla \Phi_j \ dx \end{bmatrix}_{i,j \in \{1,2,3\}}.
\end{equation}
Note that the matrix $\mathcal{M}_{2}$ is symmetric and nonnegative, as a Gram matrix. \par
Now from Lemma \ref{LemHolder}, we deduce the boundedness in $C^{\lambda,\alpha}$ of the functions $\nabla \Phi_{i}$ independently of the time. By using Lemma \ref{LemEstLp} and \eqref{BorneW13}, we obtain that
\begin{equation} \label{EstNablaMu}
\| \nabla \mu \|_{L^{2}({\mathcal F}(t))}
\leq C \big( \| \omega_{0} \|_{L^{\infty}({\mathcal F}_{0})} + | \ell_{0} | + |r_{0}| + \gamma \big).
\end{equation}
Hence with \eqref{EvoMatrice} we deduce the boundedness of $(\ell',r')$, which together with \eqref{DecompP} and \eqref{EstNablaMu}, gives the claim on $\nabla p$. The claim on $\partial_{t} u$ follows by using \eqref{Euler1}.
\end{proof}
\subsection{Uniqueness: proof of Theorem \ref{UniqYudo}}
\label{section:UniqYudo}
We now turn to the core of the proof of Theorem \ref{UniqYudo}. \par
Consider $(\ell_{1},r_{1},u_{1})$ and $(\ell_{2},r_{2},u_{2})$ two solutions in the sense of Theorem \ref{ThmYudo} defined on some time interval $[0,T]$. We associate correspondingly $h_{1}$ and $h_{2}$, ${\mathcal S}_{1}$ and ${\mathcal S}_{2}$, etc.
By a standard connectedness argument, it is sufficient to prove the uniqueness on an arbitrary small time interval $[0,\tilde{T}]$, so that we allow ourselves to choose $T>0$ small.
 We let $\tau_{1}$ and $\tau_{2}$ in $C^{2}([0,T];SE(2))$ the corresponding rigid movements associated to these solutions. For each $t \in [0,T]$ we introduce $\varphi_{t}$ and $\psi_{t}$ in $\mbox{Diff}(\overline{\Omega})$ by
\begin{equation*}
\varphi_{t} := \Psi[\tau_{2}(t)] \circ \{ \Psi[\tau_{1}(t)] \}^{-1}, \ \ \psi_{t}:= \varphi_{t}^{-1},
\end{equation*}
where $\Psi$ is defined in Proposition \ref{ProDiffeos}; we have chosen $T>0$ small enough so that $\tau_{1}(t)$ and $\tau_{2}(t)$ belong to $U$ for all $t$ in $[0,T]$. 
It is easily seen that $\varphi_{t}$ is volume preserving and sends ${\mathcal F}_{1}(t)$ into ${\mathcal F}_{2}(t)$. Now we define
\begin{equation} \label{DefU2tilde}
\tilde{u}_{2}(t,x) := [d \varphi_{t} (x)]^{-1} \cdot u_{2}(t,\varphi_{t}(x)) , \ x \in {\mathcal F}_{1}(t),
\end{equation}
the pullback of $u_{2}$ by $\varphi_{t}$, which is a solenoidal vector field on ${\mathcal F}_{1}(t)$, due to $\div u_{2}(t,\cdot)=0$ and the fact that $\varphi_{t}$ is volume preserving (see e.g. \cite[Proposition 2.4]{InoueWakimoto}). We also define
\begin{equation} \label{DefPL2tilde}
\tilde{p}_{2}(t,x) := p_{2}(t,\varphi_{t}(x))  , \ x \in {\mathcal F}_{1}(t), \ \text{ and } \
\tilde{\ell}_{2} := d(\tau_{1} \circ \tau_{2}^{-1}) \cdot \ell_{2} = Q_{1} \cdot Q_{2}^{-1} \cdot \ell_{2}.
\end{equation}
Obviously,
\begin{equation*}
u_{2}(t,x) = d \varphi_{t} (\psi_{t}(x)) \cdot \tilde{u}_{2}(t,\psi_{t}(x)) \ \text{ and } \ 
p_{2}(t,x) = \tilde{p}_{2}(t,\psi_{t}(x))  \ \text{ in } {\mathcal F}_{2}(t) .
\end{equation*}
Now to write the equation satisfied by $\tilde{u}_{2}$, we compute the partial derivatives of $u_{2}$ in terms of those of $\tilde{u}_{2}$. For convenience, we simplify the notations below as follows: an exponent $i$ designates the $i$-th component of a vector; we drop the index $2$ in $u_{2}$, $\tilde{u}_{2}$, $p_{2}$, $\tilde{p}_{2}$ and the index $t$ in $\varphi_{t}$ and $\psi_{t}$. Moreover we use Einstein's repeated indices convention and omit to write the variables with the following rules (which include the case where $\alpha$ is void so that there is no partial derivative):
\begin{gather}
\nonumber
\partial_{\alpha} u = \partial_{\alpha}u(t,x), \ \
\partial_{\alpha} \tilde{u} = \partial_{\alpha}\tilde{u}(t,\psi_{t}(x)), \ \ 
\partial_{\alpha} p = \partial_{\alpha}p(t,x), \ \
\partial_{\alpha} \tilde{p} = \partial_{\alpha}\tilde{p}(t,\psi_{t}(x)), \\ 
\label{conventions1}
\partial_{\alpha} \varphi = \partial_{\alpha}\varphi (t, \psi_{t}(x)), \ \ 
\partial_{\alpha}\psi = \partial_{\alpha}\psi (t,x).
\end{gather}
From
\begin{equation*}
u^{i} =  \partial_{k} \varphi^{i} \, \tilde{u}^{k},
\end{equation*}
we deduce
\begin{equation*}
\partial_{t} u^{i} =  \partial_{k} \varphi^{i} \, \partial_{t} \tilde{u}^{k}
+ \partial_{k} \varphi^{i} \, \partial_{l} \tilde{u}^{k} \, \partial_{t} \psi^{l}
+ (\partial_{t} \partial_{k} \varphi^{i}) \tilde{u}^{k} 
+ \partial^{2}_{kl} \varphi^{i} \, \partial_{t} \psi^{l}\, \tilde{u}^{k} ,
\end{equation*}
\begin{equation*}
\partial_{j} u^{i} =  \partial_{k} \varphi^{i} \, \partial_{l} \tilde{u}^{k}\, \partial_{j} \psi^{l}
+ (\partial^{2}_{lk} \varphi^{i}) \, \partial_{j} \psi^{l} \,  \tilde{u}^{k},
\end{equation*}
\begin{equation*}
\partial_{i} p = \partial_{k} \tilde{p} \, \partial_{i} \psi^{k}.
\end{equation*}
It follows that
\begin{eqnarray*}
(u \cdot \nabla) u^{i} &=& u^{j} \, \partial_{j} u^{i} \\
 &=& u^{j} ( \partial_{k} \varphi^{i} \, \partial_{l} \tilde{u}^{k} \partial_{j} \psi^{l}
+ (\partial^{2}_{lk} \varphi^{i}) \, \partial_{j} \psi^{l} \,  \tilde{u}^{k}) \\
&=& \partial_{m} \varphi^{j} \, \tilde{u}^{m} ( \partial_{k} \varphi^{i} \, \partial_{l} \tilde{u}^{k} \partial_{j} \psi^{l} + (\partial^{2}_{lk} \varphi^{i}) \, \partial_{j} \psi^{l} \,  \tilde{u}^{k}) \\
&=& \partial_{k} \varphi^{i} \, \tilde{u}^{l} \, \partial_{l} \tilde{u}^{k} 
+   \tilde{u}^{l}  (\partial^{2}_{lk} \varphi^{i}) \,  \tilde{u}^{k},
\end{eqnarray*}
where we used that
\begin{equation*}
\partial_{m} \varphi^{j} \, \partial_{j} \psi^{l} = \delta_{ml}.
\end{equation*}
Hence the equation of $\tilde{u}$ reads
\begin{align*}
0= \partial_{t} \tilde{u}^{i} &+ \tilde{u}^{j} \, \partial_{j} \tilde{u}^{i} + \partial_{i} \tilde{p} \\
&+ (\partial_{k} \varphi^{i} - \delta_{ik}) \partial_{t} \tilde{u}^{k} %\\
+ \partial_{k} \varphi^{i} \, \partial_{l} \tilde{u}^{k} \, (\partial_{t} \psi^{l}) + ( \partial_{k} \partial_{t} \varphi^{i}) \tilde{u}^{k} 
+ (\partial^{2}_{kl} \varphi^{i}) \, (\partial_{t} \psi^{l})\, \tilde{u}^{k} \\
&+ \tilde{u}^{l}  \, \partial_{l} \tilde{u}^{k} (\partial_{k} \varphi^{i}- \delta_{ik} )
+ \tilde{u}^{l}  (\partial^{2}_{lk} \varphi^{i}) \,  \tilde{u}^{k}\\
& + \partial_{k} \tilde{p} \, (\partial_{i} \psi^{k} - \delta_{ik}).
\end{align*}
In the above equation, all the factors between parentheses are small (in $L^{\infty}$ norm) whenever $\| \varphi_{t} - \Id \|_{C^{2}(\overline{\Omega})} + \| \partial_{t} \varphi_{t}\|_{C^{1}(\overline{\Omega})}$ is small. \par 
Now we define
\begin{gather}
\label{DefHat1}
\hat{u}(t,x) := u_{1}(t,x) - \tilde{u}_{2}(t,x) \ \text{ and } \ 
\hat{p}(t,x) := p_{1}(t,x) - \tilde{p}_{2}(t,x) \text{ in } \ {\mathcal F}_{1}(t), \\
\label{DefHat2}
\hat{h} := h_{1} -h_{2}, \ \hat{\theta} := \theta_{1} -\theta_{2}, \ 
\hat{\ell} := \ell_{1} -\tilde{\ell}_{2} \ \text{ and } \ \hat{r} := r_{1} -r_{2}.
\end{gather}
We deduce that
\begin{equation} \label{EulerDiff}
\partial_{t} \hat{u} + (u_{1} \cdot \nabla) \hat{u} + (\hat{u} \cdot \nabla) \tilde{u}_{2} + \nabla \hat{p} = \tilde{f} \ \text{ in } \ {\mathcal F}_{1}(t),
\end{equation}
with 
\begin{multline} \label{FTilde}
\tilde{f}^{i} =  (\partial_{k} \varphi^{i} - \delta_{ik}) \partial_{t} \tilde{u}_{2}^{k} 
+ \partial_{k} \varphi^{i} \, \partial_{l} \tilde{u}_{2}^{k} \, (\partial_{t} \psi^{l}) + ( \partial_{k} \partial_{t} \varphi^{i}) \tilde{u}_{2}^{k} 
+ (\partial^{2}_{kl} \varphi^{i}) \, (\partial_{t} \psi^{l})\, \tilde{u}_{2}^{k}  \\
+ \tilde{u}_{2}^{l}  \, \partial_{l} \tilde{u}_{2}^{k} (\partial_{k} \varphi^{i}- \delta_{ik} )
+ (\partial^{2}_{lk} \varphi^{i}) \,  \tilde{u}^{l} \,  \tilde{u}^{k}  
+ \partial_{k} \tilde{p}_{2} \, (\partial_{i} \psi^{k} - \delta_{ik}).
\end{multline}
Now we proceed by an energy estimate. Multiplying \eqref{EulerDiff} by $\hat{u}$ and integrating over ${\mathcal F}_{1}(t)$, we deduce
\begin{equation} \label{EnergieBase}
\int_{{\mathcal F}_{1}(t)} (\partial_{t} \hat{u} + (u_{1} \cdot \nabla) \hat{u}) \cdot \hat{u} \, dx
+ \int_{{\mathcal F}_{1}(t)} \hat{u} \cdot (\hat{u} \cdot \nabla) \tilde{u}_{2} \, dx
+ \int_{{\mathcal F}_{1}(t)} \hat{u} \cdot \nabla \hat{p} \, dx
=  \int_{{\mathcal F}_{1}(t)} \hat{u} \cdot \tilde{f} \, dx.
\end{equation}
Concerning the first term on the left hand side, using that ${\mathcal F}_{1}(t)$ is transported by the flow associated to $u_{1}$, we infer that
\begin{equation*}
\int_{{\mathcal F}_{1}(t)} (\partial_{t} \hat{u} + (u_{1} \cdot \nabla) \hat{u}) \cdot \hat{u} \, dx
= \frac{d}{dt} \int_{{\mathcal F}_{1}(t)} \frac{|\hat{u}|^{2}}{2} \, dx.
\end{equation*}
For what concerns the second term in \eqref{EnergieBase}, we use Proposition \ref{PropAPrioriEstimates}: there exists a constant $C>0$ such that for any $q \in [2,\infty)$ one has
\begin{equation*}
\| \nabla \tilde{u}_{2} \|_{L^{q}({\mathcal F}_{1}(t))} \leq C q \big( \| \omega_{0} \|_{L^{\infty}({\mathcal F}_{0})} + | \ell_{0} | + |r_{0}| + \gamma \big).
\end{equation*}
It follows that for some $C_{0} = C(\| \omega_{0} \|_{L^{\infty}({\mathcal F}_{0})} + | \ell_{0} | + |r_{0}| + \gamma)$, one has
\begin{eqnarray*}
\left| \int_{{\mathcal F}_{1}(t)} \hat{u} \cdot (\hat{u} \cdot \nabla) \tilde{u}_{2} \, dx \, \right|
\leq \| \nabla \tilde{u}_{2} \|_{L^q} \| \hat{u}^{2} \|_{L^{q'}} 
\leq C_{0} q \, \| \hat{u}^{2} \|_{L^{2}}^{\frac{2}{q'}}  .
\end{eqnarray*}
Let us turn to the third term in \eqref{EnergieBase}. We first note that, due to \eqref{PsiVolumePreserving} and \eqref{EqPsiTau}, one has $\div \hat{u}=0$ in ${\mathcal F}_{1}(t)$, $u_{1} \cdot n =\tilde{u}_{2} \cdot n=0$ on $\partial \Omega$, and 
\begin{equation*}
\tilde{u}_{2}(t,x) \cdot n_{1}(t,x) 
= (u_{2} \cdot n_{2}) (t, \tau_{2} \circ \tau_{1}^{-1}(x)) 
=(\tilde{\ell}_{2} + r_{2}(x-h_{1}(t))^{\perp}) \cdot n_{1}(t,x)
\text{ on } \ \partial {\mathcal S}_{1}(t),
\end{equation*}
where $n_{i}$ is the normal on $\partial {\mathcal S}_{i}(t)$, $i=1,2$.
It follows that
\begin{eqnarray}
\nonumber
\int_{{\mathcal F}_{1}(t)} \hat{u} \cdot \nabla \hat{p} \, dx 
&=& \int_{\partial {\mathcal S}_{1}(t)} \hat{p} (\hat{u} \cdot n_{1}) \, d \sigma \\
\nonumber
&=& \int_{\partial {\mathcal S}_{1}(t)} \hat{p} (\hat{\ell} + \hat{r} (x-h_{1}(t))^{\perp}) \cdot n_{1} \, d \sigma \\
\nonumber
&=& \begin{pmatrix} \hat{\ell} \\ \hat{r} \end{pmatrix} \cdot
\int_{\partial {\mathcal S}_{1}(t)} \hat{p}  \begin{pmatrix} n_{1} \\ (x-h_{1}(t))^{\perp} \cdot n_{1} \end{pmatrix}  \, d \sigma  \\
\label{resultantedepression}
&=& \begin{pmatrix} \hat{\ell} \\ \hat{r} \end{pmatrix} \cdot
\begin{pmatrix} m \hat{\ell}'  + m \hat{r} \tilde{\ell}_{2}^{\perp} \\ {\mathcal J} \hat{r}' \end{pmatrix} 
= \frac{1}{2} \frac{d}{dt} \big( m |\hat{\ell}|^{2} + {\mathcal J} |\hat{r}|^{2} \big) - m \hat{r} \hat{\ell} \cdot \tilde{\ell}_{2}^{\perp}.
\end{eqnarray}
We used that
\begin{equation*}
m \tilde{\ell}_{2} ' = \int_{\partial {\mathcal S}_{1}(t)} \tilde{p}_{2} n_{1} \, d \sigma + m \hat{r} \tilde{\ell}_{2}^{\perp}, 
\quad {\mathcal J}  \tilde{r}_{2} ' = \int_{\partial {\mathcal S}_{1}(t)} \tilde{p}_{2} (x-h_{1}(t))^{\perp} \cdot n_{1} \, d \sigma.
\end{equation*}
We estimate the last term in \eqref{resultantedepression} by
\begin{equation}
\label{termerelou}
| m \hat{r} \hat{\ell} \cdot \tilde{\ell}_{2}^{\perp} | \leq C(\ell_{0},r_{0},u_{0}) [|\hat{\ell}|^{2} +  |\hat{r}|^{2}].
\end{equation}
Concerning the right hand side in \eqref{EnergieBase}, we see that 
\begin{multline} \label{EstRHS}
\left| \int_{{\mathcal F}_{1}(t)} \hat{u} \cdot \tilde{f} \, dx \, \right|
\leq C \| \hat{u}(t) \|_{L^{2}({\mathcal F}_{1}(t))} \,
\Big[ \| \varphi_{t} - \Id \|_{C^{2}(\overline{\Omega})} + \| \partial_{t} \varphi_{t} \|_{C^{1}(\overline{\Omega})} \Big] \\
\times \Big( 1+ \| \partial_{t} \tilde{u}_{2}(t) \|_{L^{2}({\mathcal F}_{1}(t))} 
+ \| \tilde{u}_{2}(t) \|^{2}_{H^{1}({\mathcal F}_{1}(t))} 
+ \| \nabla \tilde{p}_{2}(t) \|_{L^{2}({\mathcal F}_{1}(t))} \Big).
\end{multline}
Using Corollaries \ref{CorEstPsi} and \ref{CorAPrioriEstimates} we obtain
\begin{equation*}
\left| \int_{{\mathcal F}_{1}(t)} \hat{u} \cdot \tilde{f} \, dx \, \right| \leq 
C (\Psi, \ell_{0},r_{0},u_{0}) \, \| \hat{u}(t) \|_{L^{2}({\mathcal F}_{1}(t))} \, 
\big( \| (\hat{h},\hat{\theta})(t) \|_{\R^{3}}  + \| (\hat{\ell},\hat{r})(t) \|_{\R^{3}} \big).
\end{equation*}
Summing up, we obtain that for any $q \in [2,\infty)$:
\begin{equation*}
\frac{d}{dt} \left( \| \hat{u}\|^{2}_{L^{2}({\mathcal F}_{1}(t))} + |\hat{\ell}|^{2} + |\hat{r}|^{2} \right) 
\leq C_{0} \left( q \| \hat{u}\|^{\frac{2}{q'}}_{L^{2}({\mathcal F}_{1}(t))} + \| \hat{u}\|^{2}_{L^{2}({\mathcal F}_{1}(t))} + |\hat{\ell}|^{2} + |\hat{r}|^{2}  + |\hat{h}|^{2} + |\hat{\theta}|^{2} \right).
\end{equation*}
Concerning the solid movement, we have
\begin{equation}
\label{solid}
| \hat{h}' | = | \ell_{1} - \ell_{2} | \leq | \ell_{1} - \tilde{\ell}_{2} | +  | \ell_{2} - \tilde{\ell}_{2} | \leq C ( |\hat{\ell}| + |\hat{\theta}|),
\end{equation}
so
\begin{equation*}
\frac{d}{dt} \left( |\hat{h}|^{2} + |\hat{\theta}|^{2} \right) \leq C \big(|\hat{\ell}|^{2} + |\hat{r}|^{2}  + |\hat{h}|^{2} + |\hat{\theta}|^{2}\big).
\end{equation*}
Hence we obtain that
\begin{eqnarray*} 
\frac{d}{dt} \left( \| \hat{u}\|^{2}_{L^{2}({\mathcal F}_{1}(t))} + |\hat{\ell}|^{2} + |\hat{r}|^{2} + |\hat{h}|^{2} + |\hat{\theta}|^{2} \right) 
& \leq C_{1} \left( q \| \hat{u}\|^{\frac{2}{q'}}_{L^{2}({\mathcal F}_{1}(t))} + |\hat{\ell}|^{2} + |\hat{r}|^{2}  + |\hat{h}|^{2} + |\hat{\theta}|^{2} \right) \\
& \leq C_{1} q \left( \| \hat{u}\|^{2}_{L^{2}({\mathcal F}_{1}(t))} + |\hat{\ell}|^{2} + |\hat{r}|^{2} + |\hat{h}|^{2} + |\hat{\theta}|^{2} \right)^\frac{1}{q'},
\end{eqnarray*}
by considering $T$ sufficiently small so that the parenthesis in the right hand side is not larger than $1$. 
Since the unique solution of $y' = N y^\frac{1}{q'}$ with $y(0)= \varepsilon>0$ and $N >0$ is given by
\begin{equation*}
y(t) = \left[ \frac{Nt}{q} + \varepsilon^{\frac{1}{q}} \right]^{q},
\end{equation*}
a comparison argument proves that 
\begin{equation*}
\| \hat{u} \|_{L^{2}}^{2} + | \hat{\ell} |^{2} +  |\hat{r}|^{2}  + |\hat{h}|^{2} + |\hat{\theta}|^{2} \leq (C_{1} t)^{q},
\end{equation*}
and we conclude that $\hat{h}=0$, $\hat{\theta}=0$ and $\hat{u}=0$ for $t< 1/C_{1}$ by letting $q \rightarrow +\infty$. 
%
%
%
%
%
%
%%%%%%%%%%%%%%%%%%%%%%%%%%%%%%%%%%%%%%%%%%%%%%%%%%%%%%%%%%%%%%%%%%%%%%%%%%%%%%%%%%%%%%%%%%%%%%%%%%%%%%%%%%%%
%
%
%
%
%
%
\section{Proof of Theorem \ref{UniqLeray}}
\label{Section:Leray}
We now turn to the viscous system.
\subsection{A priori estimates}
\label{Section:PrioriLeray} 
We begin by giving a priori estimates on a solution given by Theorem \ref{ThmLeray}. 
Therefore we assume in the sequel that $(\ell,r,u)$ is a solution as given by Theorem \ref{ThmLeray} on $[0,T]$, $T>0$. 
Let us call ${\mathcal F}(t)$ and ${\mathcal S}(t)$ the corresponding fluid and solid domains, $h$, $\theta$ the associated center of mass and angle, and $\overline{u}$ given by \eqref{globally}. We also introduce
\begin{equation} \label{DefRho2}
\rho (t,x) = \rho_{{\mathcal S}}(t,x) := \rho_{{\mathcal S}_{0}} ( (\tau^{\ell,r} (t,\cdot))^{-1}  (x) )
\text{ in } {\mathcal S}(t)
\ \text{ and } \ 
\rho(t,x) = \rho_{{\mathcal F}}= 1 \text{ in } {\mathcal F}(t),
\end{equation}
and
\begin{equation*}
u_{{\mathcal S}}(t,x) := \ell(t) + r(t) (x-h(t))^{\perp}.
\end{equation*}
We will also use, for $T>0$,  the notation 
\begin{equation*}
{\mathcal F}_T := \cup_{t \in (0,T)} \{ t \} \times {\mathcal F}(t) .
\end{equation*}
Moreover we assume that $\dist ({\mathcal S}(t),\partial \Omega) > 0 $ on $[0,T]$. \par
\ \par
The first a priori estimate is the following.
\begin{Lemma} \label{penh}
There holds
\begin{equation*}
((u \cdot \nabla) u ,u ) \in L^{\frac{4}{3}}({\mathcal F}_T ,  \R^{4}).
\end{equation*}
\end{Lemma}
\begin{proof}
The proof is left to the reader as it follows classically from the boundedness of $\Omega$, from the H\"older inequality and from Sobolev embeddings.
\end{proof}
The second a priori estimate uses the smoothing effect induced by the viscosity.
\begin{Proposition} \label{PropAPrioriEstimatesNS}
There holds
\begin{equation*}
t u  \in  L^{\frac{4}{3}} (0,T; W^{2,\frac{4}{3}} ({\mathcal F}(t)) ), \quad 
( t \partial_{t} {u}, t \nabla {p})  \in L^{\frac{4}{3}} ({\mathcal F}_T ; \R^{4}).
\end{equation*}
\end{Proposition}
The proof of Proposition \ref{PropAPrioriEstimatesNS} is rather lengthy. Therefore we first give a sketch of proof before to go into the details.
\begin{proof}[Sketch of proof of Proposition \ref{PropAPrioriEstimatesNS}]
The proof relies in a crucial way on the following auxiliary system with unknown $(\mathfrak{l},\mathfrak{r},v)$:
\begin{gather}
\label{S1}
\frac{\partial v}{\partial t} -  \Delta v + \nabla q = g \ \text{ for } \ x \in \mathcal{F}(t), \\
\label{S2}
\div v = 0 \ \text{ for } \  x \in \mathcal{F}(t) ,  \\
\label{S3}
v =  v_\mathcal{S} \ \text{ for } \ x\in \partial \mathcal{S}(t),  \\
\label{S4}
v =  0 \ \text{ for } \ x \in \partial \Omega,  \\
\label{Solide1S}
m  \mathfrak{l} ' (t) = - \int_{ \partial \mathcal{S} (t)} {\T (v,p)} n \, d \sigma + m {g}_1,  \\
\label{Solide2S}
\mathcal{J}    \mathfrak{r}' (t) =   - \int_{ \partial   \mathcal{S} (t)} {\T (v,p)}  n \cdot (x-  h (t) )^\perp \, d \sigma +  \mathcal{J} {g}_2 , \\
\label{vS}
v_{{\mathcal S}} (t,x) := \mathfrak{l} + \mathfrak{r} ( x- {h}(t))^{\perp}, 
\end{gather}
where $g$, ${g}_1$ and ${g}_2$ are some source terms and where the fluid and solid domains $\mathcal{F}(t)$ and $\mathcal{S}(t)$ are prescribed and therefore not unknown. Actually, $\mathcal{F}(t)$ and $\mathcal{S}(t)$ are associated to the solution  $(\ell,r,u)$ above. We keep the notation
\begin{equation*}
h(t) = \int_{0}^{t} \ell \ \text{ and } \ u_{{\mathcal S}}(t,x) := \ell(t) + r(t) (x-h(t))^{\perp}.
\end{equation*}
Let us now explain how this system enters into the game. We define
\begin{equation} \label{checkthat}
 v := t u , \quad  q := t p , \quad  \mathfrak{l} := t \ell ,\quad \text{ and }  \mathfrak{r} := t r .
\end{equation}
From the equations \eqref{NS1}-\eqref{Solide2NS} we infer that $(\mathfrak{l}, \mathfrak{r},v)$ is a  solution of \eqref{S1}-\eqref{vS}, in a weak sense which will be given in Definition \ref{DefWeakStokes}, with vanishing initial data 
and with, as source terms,
\begin{equation} \label{checkthatsource}
g:= u - t (u \cdot \nabla) u \in L^{\frac{4}{3}} ({\mathcal F}_{T}; \R^{2})
\ \text{ and } \ 
({g}_1, {g}_2) := (\ell, r) \in L^{\frac{4}{3}} (0,T; \R^{2} \times \R).
\end{equation}
The regularity of $g$ follows from Lemma \ref{penh}.
Then we have the following result about the existence of regular solutions to the system \eqref{S1}-\eqref{vS}.
\begin{Lemma} \label{maximal}
There exists a unique solution of \eqref{S1}-\eqref{Solide2S} on $[0,T]$ with vanishing initial data which satisfies
\begin{equation} \label{regv}
v \in  L^{\frac{4}{3}} (0,T; W^{2,\frac{4}{3}} ({\mathcal F}(t)) ), \quad 
( \partial_{t} v, \nabla q)  \in L^{\frac{4}{3}} ( {\mathcal F}_T; \R^{4}), \quad
(\mathfrak{l},  \mathfrak{r}) \in   W^{1,\frac{4}{3}} ((0,T); \R^{3}).
\end{equation}
\end{Lemma}
Lemma \ref{maximal} is an adaptation of \cite[Theorem 2.4]{GGH}. 
We will briefly explain how to modify the analysis in \cite{GGH} in order to prove Lemma \ref{maximal}. \par
Finally we will prove a result of uniqueness for weak solutions of the system \eqref{S1}-\eqref{vS} so that 
$(\mathfrak{l}, \mathfrak{r}, v)$ will also satisfy the estimates given by Lemma \ref{maximal}, which achieves the proof of Proposition \ref{PropAPrioriEstimatesNS}.
\end{proof}
The rest of Subsection \ref{Section:PrioriLeray} is devoted to the completion of the proof of Proposition \ref{PropAPrioriEstimatesNS}. \par
\subsubsection{Notion of weak solutions of the auxiliary system}
\label{Section:auxiliary} 
Similarly to the definition of $\overline{u}$ in  \eqref{globally}, we introduce  the vector field $g$ defined on $\Omega$ and associated to $g,g_1$ and $g_2$
 by
\begin{equation*}
\overline{g} (t,x) :=  g (t,x) \ \text{ for } \ x \in {\mathcal F}(t) 
\ \text{ and } \ 
\overline{g} (t,x) := g_1 (t) + g_2 (t) (x-h (t))^{\perp}  \ \text{ for } \ x \in {\mathcal S}(t) .
\end{equation*}
\begin{Definition} \label{DefWeakStokes}
Given $g \in L^{\frac{4}{3}}({\mathcal F}_{T})$, $({g}_1, {g}_2) \in L^{\frac{4}{3}} (0,T; \R^{2} \times \R) $, we say that 
\begin{equation} \label{reguW}
(\mathfrak{l}, \mathfrak{r}, v) \in  C ([0,T]; \R^2 \times \R) \times 
[ C([0,T]; L^2 ({\mathcal F}(t))) \cap L^2 (0,T; H^{1}({\mathcal F}(t)))]
\end{equation}
is a weak solution of \eqref{S1}-\eqref{vS} with vanishing initial data and with source term $(g,{g}_1, {g}_2)$ if defining $\overline{v}$ by
\begin{equation} \label{globally2}
\overline{v} (t,x) := v(t,x) \ \text{ for } \ x \in {\mathcal F}(t) 
\ \text{ and } \ 
\overline{v} (t,x) := v_\mathcal{S} (t,x) \ \text{ for } \ x \in {\mathcal S}(t), \\
\end{equation}
where $ v_\mathcal{S}$ is given by  \eqref{vS},
one has:
\begin{itemize}
\item the vector field $\overline{v}$ belongs to $L^{2}(0,T;H^{1}(\Omega))$ and is divergence free,
\item for any divergence free vector field $\phi \in C^\infty_{c} ([0,T] \times \Omega ;\R^2 )$ such that $D\phi (t,x)= 0 $ when $t \in [0,T]$ and $x \in \mathcal{S}(t)$, there holds:
\begin{multline} \label{Really}
- \int_{\Omega }  (\rho  \overline{v} \cdot  \phi )\vert_{t=T} 
+ \int_{(0,T) \times \Omega } \Big( \rho  \overline{v} \cdot \frac{\partial \phi}{\partial t}
- 2 D\overline{v} : D\phi \Big)
+  \int_0^T \! \! \int_{\partial  {\mathcal S} (t) } ( \phi \cdot v ) (u_{{\mathcal S}} \cdot n) \, d \sigma \, dt \\
= - \int_{(0,T) \times \Omega } \rho  \overline{g} \cdot  \phi  + \int_{0}^{T} m r_{\phi}\, \mathfrak{l} \cdot \ell^{\perp}.
\end{multline}
\end{itemize}
\end{Definition}
\ \par
Let us justify this definition by proving the following result. 
\begin{Lemma} \label{SimpliesW}
If $(\mathfrak{l}, \mathfrak{r}, v)$ is a classical solution of \eqref{S1}-\eqref{vS} with vanishing initial data and with source term $(g, {g}_1, {g}_2)$  then it is a weak solution in the sense of Definition \ref{DefWeakStokes}.
\end{Lemma}
\begin{proof}
We introduce $\phi$ as above. In particular, one can describe $\phi$ in ${\mathcal S}(t)$ as:
\begin{equation*}
\phi (t,x) = \ell_\phi (t) +  r_\phi (t) (x- h(t))^\perp \text{  for any } x \in \mathcal{S}(t).
\end{equation*}
We multiply the equation \eqref{S1} by $\phi(t,\cdot)$ and integrate over $\mathcal{F}(t)$. This yields 
\begin{equation*}
\int_{\mathcal{F}(t)} \phi  \cdot \Big( \partial_{t}  v-  \Delta {v} + \nabla {p} \Big) 
= \int_{\mathcal{F}(t)} \phi  \cdot \overline{g}.
\end{equation*}
Now we  observe that 
\begin{equation*}
\frac{d}{dt}  \int_{\mathcal{F}(t)} \phi \cdot v 
= \int_{\mathcal{F}(t)} (\partial_t \phi ) \cdot {v}
+ \int_{\mathcal{F}(t)} \phi  \cdot (\partial_t  {v}) +  \int_{\partial \mathcal{S}(t)} ( \phi \cdot v ) (u_{{\mathcal S}} \cdot n).
 \end{equation*}
and
\begin{eqnarray*}
\int_{\mathcal{F}(t)} \Big(  -  \Delta {v} + \nabla {p} \Big)  \cdot \phi
&=& 2 \int_{\mathcal{F}(t)} D\overline{v} :  D\phi -  \int_{ \partial \mathcal{S} (t)} ( {\T (v,p)} n) \cdot  \phi \, d \sigma 
\\ &=& 2 \int_{\Omega} D\overline{v} :  D\phi  + m (\mathfrak{l}' - g_1 ) \cdot \ell_\phi  + \mathcal{J} (\mathfrak{r}' - g_2 ) r_\phi ,
\end{eqnarray*}
thanks to  \eqref{Solide1S}-\eqref{Solide2S}.

Hence we get, after integrating in time, using $\rho=1$ in the fluid and $v(0,\cdot)=0$:
\begin{multline*}
  \int_{{\mathcal F}(T)} \rho (T,\cdot) \overline{v}(T,\cdot) \cdot \phi(T,\cdot) \, dx
- \int_{0}^{T} \! \! \int_{\mathcal{F}(t)} \rho (\partial_t \phi ) \cdot {v} \, dx \, dt
- \int_{0}^{T} \! \! \int_{\partial \mathcal{S}(t)}  ( \phi \cdot v )  (u_{{\mathcal S}} \cdot n) \, d \sigma \, dt \\
+ 2 \int_{(0,T) \times \Omega} D\overline{v} :  D\phi  
+ \int_{0}^{T} \big[ m (\mathfrak{l}' - g_1 ) \cdot \ell_\phi  + \mathcal{J} (\mathfrak{r}' - g_2 ) r_\phi \big]
= \int_{0}^{T} \! \! \int_{\mathcal{F}(t)} \phi  \cdot \overline{g}. 
\end{multline*}
Now using
\begin{equation} \label{RelRho}
\int_{{\mathcal S}(t)} \rho_{{\mathcal S}}(t,x) (x-h(t)) \, dx =0, \ \ 
m = \int_{{\mathcal S}(t)} \rho_{{\mathcal S}}(t,x) \, dx, \ \ 
{\mathcal J} = \int_{{\mathcal S}(t)} \rho_{{\mathcal S}}(t,x) |x-h(t)|^{2} \, dx,
\end{equation}
we deduce
\begin{equation*}
m g_1(t) \cdot \ell_\phi(t)  + \mathcal{J} g_2(t) \, r_\phi(t) = \int_{\mathcal{S}(t)} \rho (t,x) \phi(t,x) \cdot \overline{g}(t,x) \, dx. 
\end{equation*}
On another side, we see that in ${\mathcal S}(t)$:
\begin{equation*}
\partial_{t} \phi (t,x)= \ell'_{\phi}(t) + r'_{\phi}(t)(x-h(t))^{\perp} - r_{\phi}(t) \, \ell^{\perp}(t). 
\end{equation*}
Using an integration by parts in time and \eqref{RelRho}, we deduce
\begin{eqnarray*}
\int_{0}^{T} \big[ m \mathfrak{l}' \cdot \ell_\phi  + \mathcal{J} \, \mathfrak{r}' \, r_\phi \big] \, dt
&=& - \int_{0}^{T} \big[ m \, \mathfrak{l} \cdot \ell'_\phi  + \mathcal{J} \, \mathfrak{r} \, r'_\phi \big]
+ m \, \mathfrak{l}(T) \cdot \ell_\phi(T)  + \mathcal{J} \, \mathfrak{r}(T) \, r_\phi(T) \\
&=& - \int_{0}^{T} \! \! \int_{{\mathcal S}(t)} \rho_{S} \, v \cdot (\partial_{t} \phi) \, dx \, dt
+ \int_{0}^{T} m \, r_{\phi} \, \mathfrak{l} \cdot \ell^{\perp} \\ 
&& \quad \quad + \int_{{\mathcal S}(T)} \rho(T,\cdot) \overline{v}(T,\cdot) \cdot \phi(T,\cdot) \, dx .
\end{eqnarray*}
Adding the equalities above, we easily obtain \eqref{Really}.
\end{proof}
Let us now prove that $(\mathfrak{l}, \mathfrak{r}, v)$ given by \eqref{checkthat} is a weak solution of the auxiliary system.
\begin{Lemma}
\label{fortdonnefaible}
Let $(\ell,r,u)$ a solution of \eqref{NS1}-\eqref{Solide2NS} as given by Theorem \ref{ThmLeray}, such that $\dist ({\mathcal S}(t),\partial \Omega) > 0 $ on $[0,T]$. Then $(\mathfrak{l}, \mathfrak{r}, v)$ given by \eqref{checkthat} is a weak solution in the sense of Definition
\ref{DefWeakStokes} with source terms given by \eqref{checkthatsource}.
\end{Lemma}
\begin{proof}
First we easily verify that $(\mathfrak{l}, \mathfrak{r}, v)$ satisfies \eqref{reguW} and that $\overline{v}$ defined by  \eqref{globally2} belongs to $L^{2}(0,T;H^{1}(\Omega))$ and is divergence free. Therefore it only remains to verify \eqref{Really}. \par
We consider a divergence free vector field $\phi \in C^\infty_{c} ([0,T] \times \Omega ;\R^2 )$ such that $D\phi (t,x)= 0 $ when $t \in [0,T]$ and $x \in \mathcal{S}(t)$.  We first apply Definition \ref{DefSolusLeray} to the test function $t \phi $ (instead of $\phi$). This yields:
\begin{multline}  \label{burg}
- \int_{\Omega}  (\rho \, \overline{v} \cdot \phi )\vert_{t=T}
+ \int_{(0,T) \times \Omega} \rho \, \overline{u} \cdot \phi
+ \int_{(0,T) \times \Omega} \rho \, \overline{v} \cdot \frac{\partial \phi}{\partial t} \\
+ \int_0^T t \int_{\mathcal{F}(t)} \overline{u} \otimes \overline{u} : D\phi 
- 2 \int_{(0,T) \times \Omega} D\overline{v} : D\phi = 0.
\end{multline}
Then we use an integration by parts to get 
\begin{equation*} 
\int_{\mathcal{F}(t)} \overline{u} \otimes \overline{u} :  D\phi 
= - \int_{\mathcal{F}(t)} \big( (u \cdot \nabla) u \big) \cdot \phi 
+  \int_{\partial \mathcal{S}(t)} ( \phi \cdot \overline{u})  (u_{{\mathcal S}} \cdot n) \, d \sigma.
\end{equation*}
Hence the sum of the second and of the fourth term of \eqref{burg} can be recast as follows:
\begin{equation*} 
\int_{(0,T) \times \Omega } \rho \, \overline{u} \cdot  \phi + \int_0^T t \int_{\mathcal{F}(t)} \overline{u}\otimes \overline{u} : D\phi 
= \int_{(0,T) \times \Omega } \rho \, \overline{g} \cdot  \phi 
+  \int_0^T \! \! \int_{\partial {\mathcal S} (t)} (\phi \cdot v) (u_{{\mathcal S}} \cdot n) \, d \sigma \, dt .
\end{equation*}
Then it only suffices to observe that the last term of \eqref{Really} vanishes when $(\mathfrak{l}, \mathfrak{r}, v)$ is given by \eqref{checkthat} to conclude the proof.
\end{proof}
\subsubsection{Proof of Lemma \ref{maximal}}
\label{MaxiM} 
We now turn to the proof of Lemma \ref{maximal} which is an adaptation of \cite[Theorem 2.4]{GGH}.
Therefore we only highlight the differences with the claim of \cite[Theorem 2.4]{GGH}.
Actually \cite[Theorem 2.4]{GGH} is given for the three-dimensional case but also holds true for the two-dimensional case with the same proof. Also, another difference is that \cite[Theorem 2.4]{GGH} deals with the ``Navier-Stokes + Solid'' system. 
However their proof works as well for the system \eqref{S1}-\eqref{Solide2S}. 
It is also interesting to mention that the proof of \cite[Theorem 2.4]{GGH} uses the same kind of change of variables that we introduce in Proposition \ref{ProDiffeos}. 
Thanks to this change of variable  we are led to consider the following system: 
\begin{gather}
\label{S1ch}
\frac{\partial \tilde{v}}{\partial t} -  \Delta \tilde{v} + \nabla \tilde{q} =  \tilde{g}  \ \text{ for } \ x \in \mathcal{F}_0 , \\
\label{S2ch}
\div \tilde{v} = 0 \ \text{ for } \  x \in \mathcal{F}_0 ,  \\
\label{S3ch}
\tilde{v} =  \tilde{v}_\mathcal{S} \ \text{ for } \ x\in \partial \mathcal{S}_0 ,  \\
\label{S4ch}
\tilde{v} =  0 \ \text{ for } \ x\in \partial \Omega,  \\
\label{Solide1Sch}
m  \tilde{\mathfrak{l}}' (t) = - \int_{ \partial \mathcal{S}_0} \T(\tilde{v} ,\tilde{q}) n \, d \sigma  + \tilde{g}_1 ,  \\
\label{Solide2Sch}
\mathcal{J}   \tilde{\mathfrak{r}}' (t) =   - \int_{ \partial   \mathcal{S}_0} \T(\tilde{v} ,\tilde{q}) n \cdot (x-  h (t) )^\perp \, d \sigma + \tilde{g}_2 ,
\end{gather}
where $\tilde{g} \in L^{\frac{4}{3}} ((0,T) \times {\mathcal F}_0; \R^{2})$ and $(\tilde{g}_1, \tilde{g}_2) \in L^{\frac{4}{3}} ( 0,T ; \R^{2} \times \R)$, 
\begin{equation*}
\tilde{v}_\mathcal{S} := \tilde{\mathfrak{l}} +  \tilde{\mathfrak{r}} (x - h(t))^\perp ,
\end{equation*}
and with vanishing initial data. \par
This system appears when we use the change of variable 
\begin{equation*}
\tilde{v}(t,x) := [d \varphi_{t} (x)]^{-1} \cdot v(t,\varphi_{t}(x)) , \ x \in {\mathcal F}(t) \ \text{ and } \ \tilde{\mathfrak{l}} := [d \varphi_{t} (x)]^{-1} \cdot \mathfrak{l}, \ \tilde{\mathfrak{r}}=\mathfrak{r}, 
\end{equation*}
where $\varphi_{t} = \Psi[\tau^{\ell,r}(t)]$, $\Psi$ being defined in Proposition \ref{ProDiffeos}, and when we put the error terms resulting from the change of variable in the right hand side as in \eqref{EulerDiff}.
Then one can look for a solution of the original system by a fixed point scheme. \par
For the system \eqref{S1ch}-\eqref{Solide2Sch} the maximal regularity result \cite[Theorem 4.1]{GGH} can be straightforwardly adapted into:
\begin{Lemma} \label{maximalstokes}
Let $q \in (1,+ \infty)$ and $T>0$. 
For all $\tilde{g}$ in $ L^{q} ((0,T) \times {\mathcal F}_0; \R^{2})$, for all $(\tilde{g}_1, \tilde{g}_2)$ in $L^{q}((0,T); \R^{3}) $, 
there exists a unique solution of \eqref{S1ch}-\eqref{Solide2Sch} on $[0,T]$ with vanishing initial data satisfying
\begin{equation} \label{regtildev}
\tilde{v} \in L^{q} (0,T ; W^{2,q} ( {\mathcal F}_0 ) ), \quad 
(\partial_{t} \tilde{v}, \nabla \tilde{q})  \in L^{q} ( (0,T) \times {\mathcal F}_0 ; \R^{4}), \quad
(\tilde{\mathfrak{l}},  \tilde{\mathfrak{r}}) \in   W^{1,q} ( (0,T) ; \R^{3})   .
\end{equation}
\end{Lemma}
The proof of this lemma combines maximal regularity of the Stokes problem with inhomogeneous Dirichlet boundary conditions and some added mass effects, cf. \cite[Section 4]{GGH} and the book of Galdi \cite{Galdi}. Then Lemma \ref{maximal} can be deduced from Lemmas \ref{penh} and \ref{maximalstokes} (with $q=\frac43$) using the same fixed point procedure as in \cite[Sections 5--7]{GGH}. One gets a solution for small time $T$, and a solution defined on a larger time interval by gluing together such pieces of solutions. \par
\subsubsection{Uniqueness for the auxiliary system}
\label{WS} 
Our next step toward the proof of Proposition \ref{PropAPrioriEstimatesNS} is the following uniqueness result for weak solutions of the auxilliary system.
\begin{Lemma} \label{UniciteFaible}
Let  $g \in L^{\frac{4}{3}}({\mathcal F}_{T})$, $({g}_{1}, g_{2}) \in  L^{\frac{4}{3}} ( (0,T); \R^{3} )$,  $(\mathfrak{l}_1 ,\mathfrak{r}_1 ,v_1 ) $ and $(\mathfrak{l}_2,\mathfrak{r}_2,v_2)$ two weak solutions in the sense of Definition~\ref{DefWeakStokes} of \eqref{S1}-\eqref{vS} with vanishing initial data and with source terms $g, g_{1}, g_{2}$.
Then  $(\mathfrak{l}_1 ,\mathfrak{r}_1 ,v_1 ) = (\mathfrak{l}_2,\mathfrak{r}_2,v_2)$.
\end{Lemma}
\begin{proof}[Proof of Lemma \ref{UniciteFaible}]
We introduce $\overline{v}_{1}$ and $\overline{v}_{2}$ by \eqref{globally2}.
We define
\begin{equation*}
\hat{v}(t,x) := \overline{v}_{1}(t,x) - \overline{v}_{2}(t,x) \ \text{ in } \ \Omega, \
\hat{\mathfrak{l}} := \mathfrak{l}_{1} - \mathfrak{l}_{2} \ \text{ and } \ 
\hat{\mathfrak{r}} := \mathfrak{r}_{1} - \mathfrak{r}_{2},
\end{equation*}
so that
\begin{equation*}
\hat{v}= \hat{\mathfrak{l}} + \hat{\mathfrak{r}} (x-h)^{\perp} \text{ in } {\mathcal S}(t).
\end{equation*}
We introduce a test function $\phi$ as in Definition \ref{DefWeakStokes}, apply \eqref{Really} to $(\mathfrak{l}_{1},\mathfrak{r}_{1},v_{1})$ and $(\mathfrak{l}_{2},\mathfrak{r}_{2},v_{2})$ and make the difference of the two. We obtain
\begin{equation} \label{DiffWS}
- \int_{\Omega }  (\rho {\hat{v}} \cdot  \phi )\vert_{t=T} 
+ \int_{(0,T) \times \Omega } \Big( \rho  \hat{v} \cdot \frac{\partial \phi}{\partial t} 
- 2 D\hat{v} : D\phi \Big)
+  \int_0^T \! \! \int_{\partial  {\mathcal S} (t) } ( \phi \cdot \hat{v} ) (u_{{\mathcal S}} \cdot n) \, d \sigma \, dt
=    \int_{0}^{T} m r_{\phi}\, \hat{\mathfrak{l}} \cdot \ell^{\perp}.
\end{equation}
Now after a standard regularization procedure, we can take $\phi=\hat{v}$ in \eqref{DiffWS}. We infer
\begin{equation} \label{EstEWS}
- \frac{1}{2} \int_{\Omega }  \rho |{\hat{v}}(T,\cdot)|^{2} 
- 2 \int_{(0,T) \times \Omega } |D\hat{v}|^{2} 
+  \int_0^T \! \! \int_{\partial  {\mathcal S} (t) } |\hat{v}|^{2} (u_{{\mathcal S}} \cdot n) \, d \sigma \, dt
=  \int_{0}^{T} m \hat{\mathfrak{r}}\, \hat{\mathfrak{l}} \cdot \ell^{\perp}.
\end{equation}
Using the boundary conditions on $\hat{v}$ and the boundedness of $(\ell,r)$, one easily sees that
\begin{equation*}
\left| \int_0^T \! \! \int_{\partial {\mathcal S}(t)} |\hat{v}|^{2} (u_{{\mathcal S}} \cdot n) \, d \sigma \, dt \right|
\leq C \int_0^T \big( |\hat{\mathfrak{l}}(t)|^{2} + | \hat{\mathfrak{r}} (t) |^{2} \big) \, dt .
\end{equation*}
Also,
\begin{equation*}
\left| \int_{0}^{T} m \hat{\mathfrak{r}}\, \hat{\mathfrak{l}} \cdot \ell^{\perp} \right|
\leq  C \int_0^T \big( | \hat{\mathfrak{l}} (t) |^{2} + |  \hat{\mathfrak{r}} (t) |^{2} \big) \, dt .
\end{equation*}
Hence
\begin{equation*}
m |\hat{\mathfrak{l}}(T)|^{2} + {\mathcal J} |\hat{\mathfrak{r}}(T)|^{2}  + \| \hat{v}(T) \|^{2}_{L^{2}({\mathcal F}(T))} 
\leq C \int_{0}^{T} |(\hat{\mathfrak{l}},\hat{\mathfrak{r}})(t) |^{2} \, dt.
\end{equation*}
So Gronwall's lemma finishes the proof.
\end{proof}
\subsubsection{End of the proof of Proposition \ref{PropAPrioriEstimatesNS}}
Let us now complete the proof of Proposition \ref{PropAPrioriEstimatesNS}. \par
According to Lemma \ref{fortdonnefaible},
$(\mathfrak{l}, \mathfrak{r}, v)$ given by \eqref{checkthat} is a weak solution of the auxiliary system in the sense of Definition
\ref{DefWeakStokes} with source terms given by \eqref{checkthatsource}.
On the other hand  Lemma \ref{UniciteFaible} provides a strong solution of the same system.
According to Lemma \ref{SimpliesW} this strong solution is also a weak solution.
Let us stress in particular that the regularity in \eqref{regtildev} with $q=\frac43$ implies the  regularity in \eqref{reguW}.
According to Lemma \ref{UniciteFaible}, these two weak solutions are equal. Therefore $(\mathfrak{l}, \mathfrak{r}, v)$ given by \eqref{checkthat} satisfy \eqref{regtildev} with $q=\frac43$, which implies Proposition \ref{PropAPrioriEstimatesNS}.
\subsection{Uniqueness: proof of Theorem \ref{UniqLeray}}
\label{section:UniqLeray}
We now turn to the core of the proof of Theorem \ref{UniqLeray}. \par
We consider $(\ell_{1},r_{1},u_{1})$ and $(\ell_{2},r_{2},u_{2})$ two solutions in the sense of Theorem \ref{ThmLeray} in $[0,T]$. By the usual connectedness argument, we can suppose $T$ arbitrarily small. In particular we consider $T>0$ small enough so that no collision occurs in the time interval $[0,T]$ for both solutions. \par
Then we perform the same change of variable than in Section \ref{UniqYudo}, that is, we define $\tilde{u}_{2}$ by \eqref{DefU2tilde}, and $\tilde{p}_{2}$, $\tilde{\ell}_{2}$ by \eqref{DefPL2tilde}. 
Then, dropping temporarily  the index $2$ in $u_{2}$, $\tilde{u}_{2}$, $p_{2}$, $\tilde{p}_{2}$ and the index $t$ in $\varphi_{t}$ and $\psi_{t}$, using the notations \eqref{conventions1} and  Einstein's repeated indices convention, we obtain:
\begin{multline} \nonumber 
\partial^2_{jj} v^{i}=
\partial_{j} \psi^{m} (\partial^{2}_{mk} \varphi^{i} ) \, \partial_{l} \tilde{v}^{k} \,  \partial_{j} \psi^{l}
 +  \partial_{k} \varphi^{i} \, \partial_{j} \psi^{m}  \, \partial^{2}_{ml} \tilde{v}^{k} \,  \partial_{j} \psi^{l}
 +  \partial_{k} \varphi^{i} \, \partial_{l} \tilde{v}^{k} (\partial^2_{jj} \psi^{l} ) \\
 + \partial_{j} \psi^{m} (\partial^{3}_{mlk} \varphi^{i}) \, \partial_{j} \psi^{l} \,  \tilde{v}^{k}  
 + (\partial^{2}_{lk} \varphi^{i}) \, \partial^2_{jj} \psi^{l} \,  \tilde{v}^{k}
 + (\partial^{2}_{lk} \varphi^{i}) \, \partial_{j} \psi^{l} \,  \partial_{j} \psi^{m}  \,   \partial_{m} \tilde{v}^{k} .
\end{multline}
Hence we obtain the following equation for $\tilde{u}_{2}$:
\begin{align*}
0= \partial_{t} \tilde{u}^{i} &+ \tilde{u}^{j} \, \partial_{j} \tilde{u}^{i} + \partial_{i} \tilde{p} -  \Delta \tilde{u}^{i} 
 \\
&+ (\partial_{k} \varphi^{i} - \delta_{ik}) \partial_{t} \tilde{u}^{k} %\\
+ \partial_{k} \varphi^{i} \, \partial_{l} \tilde{u}^{k} \, (\partial_{t} \psi^{l}) + ( \partial_{k} \partial_{t} \varphi^{i}) \tilde{u}^{k} 
+ (\partial^{2}_{kl} \varphi^{i}) \, (\partial_{t} \psi^{l})\, \tilde{u}^{k} \\
& + \tilde{u}^{l}  \, \partial_{l} \tilde{u}^{k} (\partial_{k} \varphi^{i}- \delta_{ik} )
+  (\partial^{2}_{lk} \varphi^{i}) \, \tilde{u}^{l} \,  \tilde{u}^{k} 
 + \partial_{k} \tilde{p} \, (\partial_{i} \psi^{k} - \delta_{ik}) \\
& -\partial_{j} \psi^{m} (\partial^{2}_{mk} \varphi^{i} ) \, \partial_{l} \tilde{u}^{k} \,  \partial_{j} \psi^{l}
 -  (\partial_{k} \varphi^{i} \partial_{j} \psi^{m} \partial_{j} \psi^{l} - \delta_{ik} \delta_{jm}\delta_{jl} ) \partial^{2}_{ml} \tilde{u}^{k} \, 
 -  \partial_{k} \varphi^{i} \, \partial_{l} \tilde{u}^{k} (\partial^2_{jj} \psi^{l} ) \\
& \quad - \partial_{j} \psi^{m} (\partial^{3}_{mlk} \varphi^{i}) \, \partial_{j} \psi^{l} \,  \tilde{u}^{k}  
 - (\partial^{2}_{lk} \varphi^{i}) \, \partial^2_{jj} \psi^{l} \,  \tilde{u}^{k}
 - (\partial^{2}_{lk} \varphi^{i}) \, \partial_{j} \psi^{l} \,  \partial_{j} \psi^{m}  \,   \partial_{m} \tilde{u}^{k} \Big] .
\end{align*}
Once again, all the factors between parentheses in the above equation are small (in $C^{1}$ norm) whenever $\| \varphi_{t} - \Id \|_{C^{3}(\overline{\Omega})} + \| \partial_{t} \varphi_{t}\|_{C^{1}(\overline{\Omega})}$ is small. \par
Now, with the same notations \eqref{DefHat1}-\eqref{DefHat2} as in Section \ref{UniqYudo}, we obtain the following equation:
\begin{equation} \label{NSDiff}
\partial_{t} \hat{u} + (u_{1} \cdot \nabla) \hat{u} + (\hat{u} \cdot \nabla) \tilde{u}_{2} + \nabla \hat{p} -  \Delta \hat{u} = \tilde{f} \ \text{ in } \ {\mathcal F}_{1}(t),
\end{equation}
with the $i$-th component of $\tilde{f}$ given by
\begin{eqnarray*}
\nonumber \tilde{f}^{i} &=&  (\partial_{k} \varphi^{i} - \delta_{ik}) \partial_{t} \tilde{u}_{2}^{k} 
+ \partial_{k} \varphi^{i} \, \partial_{l} \tilde{u}_{2}^{k} \, (\partial_{t} \psi^{l})
+ ( \partial_{k} \partial_{t} \varphi^{i}) \tilde{u}_{2}^{k} 
+ (\partial^{2}_{kl} \varphi^{i}) \, (\partial_{t} \psi^{l})\, \tilde{u}_{2}^{k}  \\
&& + \tilde{u}_{2}^{l}  \, \partial_{l} \tilde{u}_{2}^{k} (\partial_{k} \varphi^{i}- \delta_{ik} )
+ (\partial^{2}_{lk} \varphi^{i}) \, \tilde{u}_{2}^{l} \,  \tilde{u}_{2}^{k} 
+ \partial_{k} \tilde{p}_{2} \, (\partial_{i} \psi^{k} - \delta_{ik}) \\
&& - \partial_{j} \psi^{m} (\partial^{2}_{mk} \varphi^{i} ) \, \partial_{l} \tilde{u}_{2}^{k} \,  \partial_{j} \psi^{l}
-  (\partial_{k} \varphi^{i} \partial_{j} \psi^{m} \partial_{j} \psi^{l} - \delta_{ik} \delta_{jm}\delta_{jl} ) \partial^{2}_{ml} \tilde{u}_{2}^{k} \, 
-  \partial_{k} \varphi^{i} \, \partial_{l} \tilde{u}_{2}^{k} (\partial^2_{j} \psi^{l} ) \\
&& - \partial_{j} \psi^{m} (\partial^{3}_{mlk} \varphi^{i}) \, \partial_{j} \psi^{l} \,  \tilde{u}_{2}^{k}  
 - (\partial^{2}_{lk} \varphi^{i}) \, \partial^2_{jj	} \psi^{l} \,  \tilde{u}_{2}^{k}
 - (\partial^{2}_{lk} \varphi^{i}) \, \partial_{j} \psi^{l} \,  \partial_{j} \psi^{m}  \,   \partial_{m} \tilde{u}_{2}^{k} .
\end{eqnarray*}
On the other hand, the boundary conditions \eqref{NS3}-\eqref{NS4} become 
\begin{gather*}
\tilde{u}_{2} =  \tilde{\ell}_{2} (t) + {r}_{2} (t) (x-h_{1} (t))^{\perp}  \ \text{ for } \ x\in \partial \mathcal{S}_1  (t),  \\
\tilde{u}_{2} =  0 \ \text{ for } \ x\in \partial \Omega.
\end{gather*}
The solid equations \eqref{Solide1NS}-\eqref{Solide2NS} for the second solid are now recast as (writing again $n_{1}$ for the normal on $\partial {\mathcal S}_{1}$):
\begin{gather*}
m   \tilde{\ell}_{2} ' = - \int_{ \partial \mathcal{S}_{1} (t)}  {\T} ( \tilde{u}_{2} , \tilde{p}_{2} ) n_{1} \, d \sigma + m \hat{r} \tilde{\ell}_{2}^{\perp}
,  \\
\mathcal{J}  {r}_{2} ' (t) =  -  \int_{ \partial   \mathcal{S}_{1} (t)}   {\T} ( \tilde{u}_{2} , \tilde{p}_{2} ) n_{1} \cdot (x-  h_{1} (t) )^\perp \, d \sigma .
\end{gather*}
Observe that the quantities above make sense for almost every $t>0$ thanks to Proposition \ref{PropAPrioriEstimatesNS}. \par
Now we define $\hat{\ell}$, $\hat{r}$ and $\hat{u}$, $\hat{p}$, $\hat{h}$, $\hat{\theta}$ as in \eqref{DefHat1}-\eqref{DefHat2}.
Taking the difference of the equations of $\tilde{\ell}_{2}$ and $r_{2}$ with the equations for the first solid we obtain:
\begin{gather}
\label{Con1hat}
\hat{u} =  \hat{\ell} (t) + \hat{r} (t) (x-h_{1} (t))^{\perp}  \ \text{ for } \ x\in \partial \mathcal{S}_1  (t),  \\
\label{Con2hat}
\hat{u} =  0 \ \text{ for } \ x\in \partial \Omega,  \\
\label{Solide1NShat}
m   \hat{\ell} ' = - \int_{ \partial \mathcal{S}_{1} (t)} {\T} (\hat{u} ,\hat{p})  n_{1} \, d \sigma + m \hat{r} \tilde{\ell}_{2}^{\perp}
,  \\
\label{Solide2NShat}
\mathcal{J}  \hat{r} ' (t) =  -  \int_{ \partial   \mathcal{S}_{1} (t)} {\T} (\hat{u} ,\hat{p})   n_{1} \cdot (x-  h_{1} (t) )^\perp \, d \sigma .
\end{gather}
\ \par
Now we proceed by an energy estimate. Multiplying \eqref{NSDiff} by $\hat{u}$ and integrating over ${\mathcal F}_{1}(t)$, we deduce that for almost every positive $t$ (using the regularity provided by Proposition \ref{PropAPrioriEstimatesNS}):
\begin{multline} \label{NSEnergieBase}
\int_{{\mathcal F}_{1}(t)} (\partial_{t} \hat{u} + (u_{1} \cdot \nabla) \hat{u}) \cdot \hat{u} \, dx
+ \int_{{\mathcal F}_{1}(t)} \hat{u} \cdot (\hat{u} \cdot \nabla) \tilde{u}_{2} \, dx
+ \int_{{\mathcal F}_{1}(t)} \hat{u} \cdot \nabla \hat{p} \, dx
-   \int_{{\mathcal F}_{1}(t)} \hat{u} \cdot  \Delta \hat{u}  \, dx \\
=  \int_{{\mathcal F}_{1}(t)} \hat{u} \cdot \tilde{f} \, dx.
\end{multline}
Proceeding as in  Section \ref{UniqYudo}, we have 
\begin{eqnarray*}
\int_{{\mathcal F}_{1}(t)} (\partial_{t} \hat{u} + (u_{1} \cdot \nabla) \hat{u}) \cdot \hat{u} \, dx
&=& \frac{d}{dt} \int_{{\mathcal F}_{1}(t)} \frac{|\hat{u}|^{2}}{2} \, dx ,
\\ \int_{{\mathcal F}_{1}(t)} \hat{u} \cdot \nabla \hat{p} \, dx 
&=& \begin{pmatrix} \hat{\ell} \\ \hat{r} \end{pmatrix} \cdot
\int_{\partial {\mathcal S}_{1}(t)} \hat{p}  \begin{pmatrix} n_{1} \\ (x-h_{1}(t))^{\perp} \cdot n_{1} \end{pmatrix}  \, d \sigma .
\end{eqnarray*}
For the third and fourth term in \eqref{NSEnergieBase}, we have 
\begin{eqnarray*}
-  \int_{{\mathcal F}_{1}(t)} \hat{u} \cdot  \Delta \hat{u}  \, dx
&=&  2 \int_{{\mathcal F}_{1}(t)} D\hat{u} : D\hat{u}  \, dx
- \int_{\partial {\mathcal S}_{1}(t)} (D\hat{u} \cdot n_{1}) \cdot \hat{u}  \, d \sigma \\
&=& 2 \int_{{\mathcal F}_{1}(t)} D\hat{u} : D\hat{u}  \, dx
- \begin{pmatrix} \hat{\ell} \\ \hat{r} \end{pmatrix} \cdot 
\int_{\partial {\mathcal S}_{1}(t)} \begin{pmatrix} D\hat{u} \cdot n_{1} \\  (x-h_{1}(t))^{\perp} \cdot (D\hat{u} \cdot n_{1}) \end{pmatrix}  \, d \sigma,
\end{eqnarray*}
thanks to \eqref{Con1hat}-\eqref{Con2hat}.
Thus 
\begin{eqnarray}
\nonumber
\int_{{\mathcal F}_{1}(t)} \hat{u} \cdot \nabla \hat{p} \, dx  -   \int_{{\mathcal F}_{1}(t)} \hat{u} \cdot  \Delta \hat{u}  \, dx
&=&  2 \int_{{\mathcal F}_{1}(t)} D\hat{u} : D\hat{u}  \, dx - 
 \begin{pmatrix} \hat{\ell} \\ \hat{r} \end{pmatrix} \cdot
\int_{\partial {\mathcal S}_{1}(t)}  \begin{pmatrix} {\T} (\hat{u} ,\hat{p}) n_{1} \\ (x-h_{1}(t))^{\perp} \cdot {\T} (\hat{u} ,\hat{p}) n_{1} \end{pmatrix}  \, d \sigma , \\ 
\label{transfotermesbord}
&=& 2 \int_{{\mathcal F}_{1}(t)} D\hat{u} : D\hat{u}  \, dx 
+  \frac{1}{2} \frac{d}{dt} \big( m |\hat{\ell}|^{2} + {\mathcal J} |\hat{r}|^{2} \big) - m \hat{r} \hat{\ell} \cdot \tilde{\ell}_{2}^{\perp}.
\end{eqnarray}
thanks to \eqref{Solide1NShat}-\eqref{Solide2NShat}. The last term in the right hand side of \eqref{transfotermesbord} is estimated as in \eqref{termerelou}. \par
\ \par
Now for the second term in \eqref{NSEnergieBase}, we will use the following lemma.
\begin{Lemma} \label{LemInterpolation}
There exists $C >0$ such that for any $t\in (0,T)$, for any $w \in H^{1} ({\mathcal F}_1 (t)  )$ vanishing on $\partial \Omega$ and any $\varepsilon >0$, 
\begin{equation*}
\| w \|_{L^{4} ({\mathcal F}_1 (t)  )} \leqslant 
\frac{C}{\varepsilon} \|  w \|_{L^{2} ({\mathcal F}_1 (t)  )} 
+ \varepsilon \| \nabla  w \|_{L^{2} ({\mathcal F}_1 (t)  )} .
\end{equation*}
\end{Lemma}
\begin{proof}[Proof of Lemma \ref{LemInterpolation}]
A classical interpolation argument gives that for any $t\in (0,T)$, for any $w \in H^{1} ({\mathcal F}_1 (t)  )$, 
\begin{equation*}
\| w \|_{L^{4} ({\mathcal F}_1 (t)  )} 
\leqslant  C \|  w \|_{L^{2} ({\mathcal F}_1 (t) )}^{1/2} \|  w \|_{H^{1} ({\mathcal F}_1 (t) )}^{1/2}.
\end{equation*}
Since the Poincar\'e inequality holds for $w \in H^{1} ({\mathcal F}_1 (t))$ vanishing on $\partial \Omega$, we deduce that for such $w$,
\begin{equation} \label{Interpolation2}
\| w \|_{L^{4} ({\mathcal F}_1 (t)  )} \leqslant  
C \|  w \|_{L^{2} ({\mathcal F}_1 (t) )}^{1/2} \| \nabla w \|_{L^{2} ({\mathcal F}_1 (t) )}^{1/2}.
\end{equation}
We deduce the claim.
\end{proof}
It follows that we can estimate the second term in \eqref{NSEnergieBase} by
\begin{eqnarray*}
\left| \int_{{\mathcal F}_{1}(t)} \hat{u} \cdot (\hat{u} \cdot \nabla) \tilde{u}_{2} \, dx \, \right|
\leq   \| \nabla \tilde{u}_{2} \|_{L^2}  \, \| \hat{u} \|_{L^{4}}^{2}
  \leq C  \| \nabla \tilde{u}_{2} \|_{L^2}^{2} \| \hat{u} \|^{2}_{L^{2}}  +  \frac{1}{4}  \| \nabla \hat{u} \|_{L^{2}}^{2} ,
\end{eqnarray*}
where the norms above are over ${\mathcal F}_{1}(t)$. \par
\ \par
Let us now turn to the estimate of the right hand side in \eqref{NSEnergieBase}. The estimate is given in the following lemma.
\begin{Lemma} \label{Lem:EstRHS}
For some constant $C>0$ depending on the geometry only and defining the function ${\mathcal B} \in L^{1}(0,T)$ by
\begin{multline*}
{\mathcal B}(t):= 
\| \tilde{u}_{2} \|_{L^{\infty} (0,T ; L^{2} ({\mathcal F}_{1}(t)))}  (1 +  \| \nabla \tilde{u}_{2}(t,\cdot) \|_{L^{2}({\mathcal F}_{1}(t))}) \\
+  \| \tilde{u}_{2} \|_{L^{\infty} (0,T ; L^{2} ({\mathcal F}_{1}(t)))}^{1/2}  \| \nabla \tilde{u}_{2}(t) \|^{1/2}_{L^{2}({\mathcal F}_{1}(t))} \| t \nabla \tilde{u}_{2}(t) \|_{L^{4}({\mathcal F}_{1}(t))} \\
+ \big(\| t \partial_{t} \tilde{u}_{2} \|_{L^{4/3}({\mathcal F}_{1}(t))} + \| t \tilde{u}_{2} \|_{W^{2,4/3}({\mathcal F}_{1}(t))} + \| t \nabla \tilde{p}_{2} \|_{L^{4/3}({\mathcal F}_{1}(t))} \big)^{4/3},
\end{multline*}
one has the following estimate on the right hand side:
\begin{multline} \label{Eq:EstRHS}
\left| \int_{0}^T \! \! \int_{{\mathcal F}_{1}(t)} \hat{u} \cdot \tilde{f} \, dx \, dt  \right| 
\leq \frac{1}{4} \int_{0}^T \! \! \int_{{\mathcal F}_{1}(t)} |\nabla \hat{u}|^{2} \, dx \, dt \\
+ C \int_{0}^{T} {\mathcal B}(t) \Big[ 
  \max_{\tau \in [0,t]} \| \hat{u}(\tau,\cdot) \|^{2}_{L^{2}({\mathcal F}_{1}(t))} 
+ \max_{[0,t]} | (\hat{h},\hat{\theta},\hat{\ell},\hat{r}) |^{2} 
\Big] \, dt.
\end{multline}
\end{Lemma}
\begin{proof}[Proof of Lemma \ref{Lem:EstRHS}]
In what follows, $C>0$ denotes various positive constants depending on the geometry and which can change from line to line.
We cut $\tilde{f}$ into pieces which are to be estimated separately. Precisely, we denote
\begin{equation*}
\tilde{f} = \tilde{f}_{1} + \tilde{f}_{2}+ \tilde{f}_{3} + \tilde{f}_{4} + \tilde{f}_{5},
\end{equation*}
with
\begin{align*}
\tilde{f}_{1} &:= ( \partial_{k} \partial_{t} \varphi^{i}) \tilde{u}_{2}^{k} 
+ (\partial^{2}_{kl} \varphi^{i}) \, (\partial_{t} \psi^{l})\, \tilde{u}_{2}^{k}
- \sum_{j} \Big[ \partial_{j} \psi^{m} (\partial^{3}_{mlk} \varphi^{i}) \, \partial_{j} \psi^{l} \,  \tilde{u}_{2}^{k}  
+ (\partial^{2}_{lk} \varphi^{i}) \, \partial^2_{jj} \psi^{l} \,  \tilde{u}_{2}^{k} \Big], \\
\tilde{f}_{2} &:= \partial_{k} \varphi^{i} \, \partial_{l} \tilde{u}_{2}^{k} \, (\partial_{t} \psi^{l}) 
- \sum_{j} \Big[ \partial_{j} \psi^{m} (\partial^{2}_{mk} \varphi^{i} ) \, \partial_{l} \tilde{u}_{2}^{k} \,  \partial_{j} \psi^{l}
+ \partial_{k} \varphi^{i} \, \partial_{l} \tilde{u}_{2}^{k} (\partial^2_{j} \psi^{l} ) 
+ (\partial^{2}_{lk} \varphi^{i}) \, \partial_{j} \psi^{l} \,  \partial_{j} \psi^{m}  \,   \partial_{m} \tilde{u}_{2}^{k} \Big] , \\
\tilde{f}_{3} &:= (\partial^{2}_{lk} \varphi^{i}) \, \tilde{u}_{2}^{l} \,  \tilde{u}_{2}^{k} , \\
\tilde{f}_{4} &:= \tilde{u}_{2}^{l}  \, \partial_{l} \tilde{u}_{2}^{k} (\partial_{k} \varphi^{i}- \delta_{ik} ) , \\
\tilde{f}_{5} &:= (\partial_{k} \varphi^{i} - \delta_{ik}) \partial_{t} \tilde{u}_{2}^{k} 
+ \partial_{k} \tilde{p}_{2} \, (\partial_{i} \psi^{k} - \delta_{ik}) 
- \sum_{j} (\partial_{k} \varphi^{i}\partial_{j} \psi^{m} \partial_{j} \psi^{l} - \delta_{ik} \delta_{jm} \delta_{jl}) \partial^{2}_{ml} \tilde{u}_{2}^{k}.
\end{align*}
\begin{itemize}
\item[$\bullet$] Concerning $\tilde{f}_{1}$, using Corollary \ref{CorEstPsi} we deduce that
\begin{eqnarray*}
\left| \int_{0}^T \! \! \int_{{\mathcal F}_{1}(t)} \hat{u} \cdot \tilde{f}_{1} \, dx \, dt  \right| 
&\leq& C \| \tilde{u}_{2} \|_{L^{\infty} (0,T ; L^{2} ({\mathcal F}_{1}(t)))}
\int_{0}^{T}   \max_{\tau \in [0,t]} \| \hat{u}(\tau,\cdot) \|_{L^{2}({\mathcal F}_{1}(t))}
\max_{[0,t]} | (\hat{h},\hat{\theta},\hat{\ell},\hat{r})| \, dt \\
&\leq& C  \| \tilde{u}_{2} \|_{L^{\infty} (0,T ; L^{2} ({\mathcal F}_{1}(t)))} 
\int_{0}^{T}  \big( \max_{\tau \in [0,t]} \| \hat{u}(\tau,\cdot) \|_{L^{2}({\mathcal F}_{1}(t))}^{2} 
+ \max_{[0,t]} |(\hat{h},\hat{\theta},\hat{\ell},\hat{r})|^{2} \big) \, dt.
\end{eqnarray*}
\item[$\bullet$] Concerning $\tilde{f}_{2}$, using Corollary \ref{CorEstPsi} one has for almost every $t$:
\begin{equation*}
\left| \int_{{\mathcal F}_{1}(t)} \hat{u} \cdot \tilde{f}_{2} \, dx \right| \leq C  \| \nabla \tilde{u}_{2}(t,\cdot) \|_{L^{2}({\mathcal F}_{1}(t))} |(\hat{h},\hat{\theta})(t)| \, \| \hat{u} \|_{L^{2}({\mathcal F}_{1}(t))}.
\end{equation*}
We fix 
\begin{equation*}
{\mathcal B}_{1}(t) := \| \nabla \tilde{u}_{2}(t,\cdot) \|_{L^{2}({\mathcal F}_{1}(t))} \in L^{2}(0,T) \subset L^{1}(0,T),
\end{equation*}
and have
\begin{equation*}
\left| \int_{0}^T \! \! \int_{{\mathcal F}_{1}(t)} \hat{u} \cdot \tilde{f}_{2} \, dx \, dt  \right| 
\leq C \int_{0}^{T} {\mathcal B}_{1}(t) \big(\max_{\tau \in [0,t]} \| \hat{u}(\tau,\cdot) \|_{L^{2}({\mathcal F}_{1}(t))}^{2} + \max_{[0,t]} |(\hat{h},\hat{\theta})|^{2}\big) \, dt.
\end{equation*}
\item[$\bullet$] Concerning $\tilde{f}_{3}$: one has for almost every $t>0$, using \eqref{Interpolation2}:
\begin{eqnarray*}
\left| \int_{{\mathcal F}_{1}(t)} \hat{u} \cdot \tilde{f}_{3} \, dx \right| &\leq& C  \| \tilde{u}_{2}(t,\cdot) \|^{2}_{L^{4}({\mathcal F}_{1}(t))} |(\hat{h},\hat{\theta})(t)| \, \| \hat{u} \|_{L^{2}({\mathcal F}_{1}(t))} \\
&\leq& C \| \tilde{u}_{2}(t,\cdot) \|_{L^{2}({\mathcal F}_{1}(t))}
\| \nabla \tilde{u}_{2}(t,\cdot) \|_{L^{2}({\mathcal F}_{1}(t))}
|(\hat{h},\hat{\theta})(t)| \, \| \hat{u} \|_{L^{2}({\mathcal F}_{1}(t))}.
\end{eqnarray*}
Using again the function ${\mathcal B}_{1}$, we have 
\begin{equation*}
\left| \int_{0}^T \! \! \int_{{\mathcal F}_{1}(t)} \hat{u} \cdot \tilde{f}_{3} \, dx \, dt  \right| 
\leq C  \| \tilde{u}_{2} \|_{L^{\infty} (0,T ; L^{2} ({\mathcal F}_{1}(t)))}
\int_{0}^{T} {\mathcal B}_{1}(t) \big( \max_{\tau \in [0,t]} \| \hat{u}(\tau,\cdot) \|_{L^{2}({\mathcal F}_{1}(t))}^{2}
 + \max_{[0,t]} |(\hat{h},\hat{\theta})|^{2} \big) \, dt.
\end{equation*}
\item[$\bullet$] Concerning $\tilde{f}_{4}$: we first note that thanks to Proposition \ref{PropAPrioriEstimatesNS}, we have 
\begin{equation*}
t \nabla \tilde{u}_{2} \in L^{4/3}(0,T;W^{1,4/3}({\mathcal F}_{1}(t))) \hookrightarrow L^{4/3}(0,T;L^{4}({\mathcal F}_{1}(t))).
\end{equation*}
On another side we infer from Corollary \ref{CorEstPsi} that for some constant one has
\begin{equation} \label{EstTransvase}
\| \frac{1}{t} (\partial_{k} \varphi^{i}_{t} - \delta_{ik}) \|_{C^{3}(\overline{\Omega})} \leq C
\| (\hat{\ell},\hat{r}) \|_{L^{\infty}(0,t)}.
\end{equation}
Using \eqref{Interpolation2} and \eqref{EstTransvase} we deduce that 
\begin{eqnarray*}
\left| \int_{0}^{T} \! \! \int_{{\mathcal F}_{1}(t)} \hat{u} \cdot \tilde{f}_{4} \, dx \, dt \right|
\leq C \int_{0}^{T} \| \tilde{u}_{2}(t) \|_{L^{4}({\mathcal F}_{1}(t))} 
\| t \nabla \tilde{u}_{2}(t) \|_{L^{4}({\mathcal F}_{1}(t))} \| (\hat{\ell},\hat{r}) \|_{L^{\infty}(0,t)}
\| \hat{u}(t,\cdot) \|_{L^{2}({\mathcal F}_{1}(t))} \, dt \\
\quad \leq C \int_{0}^{T} \| \tilde{u}_{2}(t) \|^{1/2}_{L^{2}({\mathcal F}_{1}(t))} 
\| \nabla \tilde{u}_{2}(t) \|^{1/2}_{L^{2}({\mathcal F}_{1}(t))} 
\| t \nabla \tilde{u}_{2}(t) \|_{L^{4}({\mathcal F}_{1}(t))} \| (\hat{\ell},\hat{r}) \|_{L^{\infty}(0,t)}
\| \hat{u}(t,\cdot) \|_{L^{2}({\mathcal F}_{1}(t))} \, dt.
\end{eqnarray*}
We introduce
\begin{equation*}
{\mathcal B}_{2}(t) :=  \| \nabla \tilde{u}_{2}(t) \|^{1/2}_{L^{2}({\mathcal F}_{1}(t))} \| t \nabla \tilde{u}_{2}(t) \|_{L^{4}({\mathcal F}_{1}(t))}  \in L^{1}(0,T),
\end{equation*}
as a product $L^{4}(0,T) \times L^{4/3}(0,T)$, and deduce
\begin{equation*}
\left| \int_{0}^{T} \! \! \int_{{\mathcal F}_{1}(t)} \hat{u} \cdot \tilde{f}_{4} \, dx \, dt \right|
\leq C  \| \tilde{u}_{2} \|_{ L^{\infty} (0,T ; L^{2} ({\mathcal F}_{1}(t)))}^{1/2} \int_{0}^{T} {\mathcal B}_{2}(t) \big[ \| (\hat{\ell},\hat{r}) \|_{L^{\infty}(0,t)}^{2}
+ \| \hat{u}(t,\cdot) \|_{L^{2}({\mathcal F}_{1}(t))}^{2} \big] \, dt.
\end{equation*}
\item[$\bullet$] Concerning $\tilde{f}_{5}$: we use again \eqref{EstTransvase} and we introduce 
\begin{equation*}
b(t) := \| t \partial_{t} \tilde{u}_{2}(t) \|_{L^{4/3}({\mathcal F}_{1}(t))} + \| t \tilde{u}_{2}(t) \|_{W^{2,4/3}({\mathcal F}_{1}(t))} +  \| t \nabla \tilde{p}_{2}(t) \|_{L^{4/3}({\mathcal F}_{1}(t))},
\end{equation*}
which belongs to $L^{4/3}(0,T)$ thanks to Proposition \ref{PropAPrioriEstimatesNS}. One deduces that
\begin{equation*}
\left| \int_{0}^{T} \int_{{\mathcal F}_{1}(t)} \hat{u} \cdot \tilde{f}_{5} \, dx \, dt \right|
\leq
C \int_{0}^{T} b(t) \| (\hat{\ell},\hat{r}) \|_{L^{\infty}(0,t)}
\| \hat{u}(t,\cdot) \|_{L^{4}({\mathcal F}_{1}(t))} \, dt .
\end{equation*}
Hence, with 
$2 b \mu \nu \leq b^{2 \alpha} \mu^{2} +  b^{2 (1-\alpha)} \nu^{2}$ for $\mu, \nu \in \R$, $b\geq0$ and $\alpha \in (0,1)$,
we deduce that
\begin{equation*}
\left| \int_{0}^{T} \! \! \int_{{\mathcal F}_{1}(t)} \hat{u} \cdot \tilde{f}_{5} \, dx \, dt \right| \leq C_{1} \int_{0}^{t} b(t)^{2/3} \| \hat{u}(t,\cdot) \|^{2}_{L^{4}({\mathcal F}_{1}(t))} \, dt 
+ C \int_{0}^{t} b(t)^{4/3}  \| (\hat{\ell},\hat{r}) \|^{2}_{L^{\infty}(0,t)} \, dt .
\end{equation*}
We specify the constant $C_{1}$ for later use. For the first term, one writes
\begin{eqnarray*}
\int_{0}^{T} b(t)^{2/3} \| \hat{u}(t,\cdot) \|^{2}_{L^{4}({\mathcal F}_{1}(t))} \, dt & \leq & C 
\int_{0}^{T} b(t)^{2/3} \| \hat{u}(t,\cdot) \|_{L^{2}({\mathcal F}_{1}(t))}  \| \nabla \hat{u}(t,\cdot) \|_{L^{2}({\mathcal F}_{1}(t))} \, dt \\
 & \leq & C \int_{0}^{T} b(t)^{4/3} \| \hat{u}(t,\cdot) \|^{2}_{L^{2}({\mathcal F}_{1}(t))}  \, dt 
+ \frac{1}{4C_{1}} \int_{0}^{T} \| \nabla \hat{u}(t,\cdot) \|^{2}_{L^{2}({\mathcal F}_{1}(t))} \, dt .
\end{eqnarray*}
So one has
\begin{equation*}
\left| \int_{0}^{T} \! \! \int_{{\mathcal F}_{1}(t)} \hat{u} \cdot \tilde{f}_{5} \, dx \, dt \right|
\leq \frac{1}{4} \int_{0}^{T} \| \nabla \hat{u}(t,\cdot) \|^{2}_{L^{2}({\mathcal F}_{1}(t))} \, dt +
C \int_{0}^{T} {\mathcal B}_{3}(t) \Big[ \| \hat{u}(t,\cdot) \|_{L^{2}({\mathcal F}_{1}(t))}^{2} + \| (\hat{\ell},\hat{r}) \|^{2}_{L^{\infty}(0,t)} \Big] \, dt ,
\end{equation*}
with ${\mathcal B}_{3}:= b(t)^{4/3} \in L^{1}(0,T)$.
\end{itemize}
\ \par
Summing up all the estimates above, we deduce \eqref{Eq:EstRHS}.
\end{proof}
\ \par
\noindent
{\it Back to the proof of Theorem \ref{ThmLeray}.}
We extend $\hat{u}(t,\cdot)$ inside ${\mathcal F}_{1}(t)$ by $\hat{\ell} + \hat{r} (x-h_{1}(t))$. We obtain that $\hat{u}(t,\cdot)$ is a $L^{2}(0,T;H^1(\Omega))$ divergence free vector field, vanishing on $\partial \Omega$.
Therefore
\begin{equation*}
\int_{{\mathcal F}_{1}(t)} |\nabla \hat{u}|^{2}\, dx \leq
\int_{\Omega} |\nabla \hat{u}|^{2} \, dx = 
2 \int_{\Omega} |D\hat{u} |^{2}  \, dx 
= 2 \int_{{\mathcal F}_{1}(t)} |D\hat{u}|^{2}  \, dx .
\end{equation*}
Now we take into account the vanishing initial condition for $(\hat{\ell},\hat{r},\hat{u})$ to deduce that for any $T>0$ sufficiently small,
\begin{equation*}
m |\hat{\ell}(T)|^{2} + {\mathcal J} |\hat{r}(T)|^{2}  + \| \hat{u}(T) \|^{2}_{L^{2}({\mathcal F}_{1}(T))} \leq C \int_{0}^{T} {\mathcal B}(t) \Big[ 
  \max_{\tau \in [0,t]} \| \hat{u}(\tau,\cdot) \|^{2}_{L^{2}({\mathcal F}_{1}(t))} 
+ \max_{[0,t]} |(\hat{h},\hat{\theta},\hat{\ell},\hat{r})(t) |^{2} 
\Big] \, dt.
\end{equation*}
Proceeding as in \eqref{solid} we get 
\begin{equation*}
\frac{d}{dt} \left( |\hat{h}|^{2} + |\hat{\theta}|^{2} \right) \leq C \big(|\hat{\ell}|^{2} + |\hat{r}|^{2}  + |\hat{h}|^{2} + |\hat{\theta}|^{2}\big).
\end{equation*}
Hence using ${\mathcal B}(t) \in L^{1}$ and Gronwall's lemma concludes the proof.
\section{Appendix. Proof of Theorem \ref{ThmYudo}}
In this appendix, we will use the letter $\eta$ for the fluid flow and $\tau$ for the solid flow. \par
To $(\ell,r) \in C^{0}([0,T];\R^{2} \times \R)$ we can associate $h^{\ell,r} , \theta^{\ell,r} , u_{{\mathcal S}}^{\ell,r} , \tau^{\ell,r} $ and ${\mathcal F}^{\ell,r}$ by \eqref{Eq:xbq}, \eqref{Defvsolide}, \eqref{Eq:rota}, \eqref{Eq:St} and \eqref{Eq:Do}. We also introduce
\begin{equation*}
\varphi^{\ell,r} := \Psi (\tau^{\ell,r}),
\end{equation*}
where $\Psi$ was defined in Lemma \ref{ProDiffeos}. We can ensure that $\tau^{\ell,r}$ belongs to the set $U$ of definition of $\Psi$ by choosing $T$ suitably small. \par
We may omit the dependence on $(\ell,r)$ on the above objects when there is no ambiguity. \par
As in Section \ref{Section:PrioriYudo} we suppose that $\partial \Omega$ has $g+1$ connected components $\Gamma_{1}, \dots , \Gamma_{g+1}$ and that $\Gamma_{g+1}$ is the outer one; and 
we denote by $\Gamma_{0}=\Gamma_{0}(t)=\partial {\mathcal S}(t)$,  by $\mathfrak{t}$ the tangent to $\partial \Omega$ and $\partial {\mathcal S}(t)$ and we define
\begin{equation*}
\gamma_{0}^{i} := \int_{\Gamma_{i}} u_{0} \cdot \mathfrak{t} \, d \sigma \ \text{ for } i=1, \dots, g \ \text{ and } \ 
\gamma_{0} := \int_{\partial {\mathcal S}_{0}} u_{0} \cdot \mathfrak{t} \, d \sigma .
\end{equation*}
We will use the following variant of Lemma \ref{LemEstLp}.
\begin{Lemma} \label{LemEstLL}
For any $R>0$, there exists $C>0$ such that if ${\mathcal S} = \tau ({\mathcal S}_{0})$ for $\tau \in SE(2)$ satisfies
\eqref{ContrainteS}
then any $u : {\Omega \setminus {\mathcal S}} \rightarrow \R^{2}$ satisfying
\begin{equation*}
	\div u = 0 \ \text{ in } \ \Omega \setminus {\mathcal S} , \quad
	u \cdot n = 0 \ \text{ on } \ \partial \Omega  \ \text{ and } \   u \cdot n =  ( \ell + r x^\perp ) \cdot  n    \text{ on } \ \partial {\mathcal S},
\end{equation*}
where $( \ell , r ) \in \R^2 \times \R$, 
verifies (setting again $\Gamma_{0}:= \partial{\mathcal S}$):
\begin{equation} \label{EllipticEstimateLL}
\| u \|_{\mathcal{LL} (\Omega \setminus {\mathcal S})} \leq C \Big( \| \curl u \|_{L^{\infty}(\Omega \setminus {\mathcal S})} + \sum_{i=0}^{g} \left| \int_{\Gamma_{i}}  u \cdot \mathfrak{t} \, d \sigma \right|      +  | \ell |   +  | r |  \Big).
\end{equation}
\end{Lemma}
\begin{proof}
It is a direct consequence of Lemma \ref{LemEstLp} and Morrey's estimates. It can also be established directly by following the lines of the proof of Lemma \ref{LemEstLp}.
\end{proof}

\ \par
\subsection{With a prescribed solid movement}
\label{Subsec:PrescribedMovement}
We first prove the following result, which concerns the Euler system with a prescribed solid movement of ${\mathcal S}(t)$ inside $\Omega$, and gives  existence of a solution as long as no collision occurs. 
\begin{Proposition}
\label{SolideImpose}
Let $T >0$ and a regular closed connected subset $\mathcal{S}_0 \subset \Omega$ and define ${\mathcal F}_{0}:= \Omega \setminus {\mathcal S}_{0}$. 
Consider $(\ell,r) \in C^{0}([0,T]; \R^{2} \times \R)$ such that
\begin{equation} \label{LoinDuBord}
\text{ for any } t \in [0,T], \ \  \dist \big( \tau^{\ell,r}(t)[{\mathcal S}_{0}], \partial \Omega \big) >0.
\end{equation}
Consider $u_0  \in C^{0}(\overline{\mathcal{F}_0};\R^{2})$ satisfying  \eqref{CondCompatibilite} and \eqref{TourbillonYudo}.
Then  the problem \eqref{Euler1}-\eqref{Euler2}-\eqref{Euler3}-\eqref{Euler4} (with ${\mathcal S}(t):=\tau^{\ell,r}({\mathcal S}_{0})$ and ${\mathcal F}(t):= \Omega \setminus {\mathcal S}(t)$) admits a unique solution 
$$u \in L^{\infty} (0,T; \mathcal{LL} (\mathcal{F}(t)) ) \cap C^{0}([0,T]; W^{1,q}({\mathcal F}(t)))], \ \ \forall q \in [1,+\infty). $$
\end{Proposition}
\noindent
{\it Proof of Proposition \ref{SolideImpose}.}
We use Schauder's fixed point theorem in order to prove the existence part. 
Let $(\ell,r)$ be fixed so that \eqref{LoinDuBord} holds. We deduce $\tau(t)$, $\varphi(t)$, ${\mathcal S}(t)$ and ${\mathcal F}(t)$ as previously.
We will also use, for $T>0$,  the notation 
\begin{equation*}
{\mathcal F}_T := \cup_{t \in (0,T)} \{ t \} \times {\mathcal F}(t) .
\end{equation*}
We let
\begin{align*}
{\mathcal C}:=\Big\{ w \in L^{\infty}((0,T) \times {\mathcal F}_{0}) \ / \ \| w \|_{L^{\infty}((0,T) \times {\mathcal F}_{0})} \leq \| \omega_{0} \|_{L^{\infty}({\mathcal F}_{0})}
 \Big\}.
\end{align*}
We endow ${\mathcal C}$ with the $L^{\infty}(0,T;L^{3}({\mathcal F}_{0}))$ topology. Note that ${\mathcal C}$ is closed and convex. \par
\ \par
Now we define ${\mathcal T}= {\mathcal T}^{\ell,r}:{\mathcal C} \rightarrow {\mathcal C}$ as follows. Given $w \in {\mathcal C}$, we define $\omega: {\mathcal F}_T \rightarrow \R^{2}$ by
\begin{equation} \label{DefOmega}
\omega(t,x) =  w(\varphi(t)^{-1}(x)),
\end{equation}
which belongs to $L^{\infty}(0,T;L^{\infty}({\mathcal F}(t)))$.
Next we define $u: {\mathcal F}_T  \rightarrow \R^{2}$ by the following system
\begin{equation} \label{Eq:vtransporte}
\left\{ \begin{array}{l}
\curl {u} = {\omega} \text{ in } {\mathcal F}_T, \\
\div {u} = 0 \text{ in } {\mathcal F}_T, \\
{u}\cdot n =0 \text{ on } [0,T] \times \partial \Omega, \\
{u}(t,x)\cdot n = u_{{\mathcal S}}(t,x)\cdot n \text{ for } t\in [0,T] \text{ and }  x \in  \partial {\mathcal S(t)}, \\
\int_{\Gamma_{i}} {u} \cdot \mathfrak{t} \, d \sigma  = \gamma_{0}^{i}  \text{ for all } i =1 \dots g, \\
\int_{\partial {\mathcal S(t)}} {u} \cdot \mathfrak{t} \, d \sigma  = \gamma_{0} ,
\end{array} \right.
\end{equation}
with $u_{{\mathcal S}}$ defined in \eqref{Defvsolide}. 
According to Lemma  \ref{LemEstLL}, $u$ belongs to $L^{\infty}(0,T; \mathcal{LL}({\mathcal F}(t)))$.\par
Consequently we can define the flow $\eta(t,x)$ associated to $u$ in a unique way. This flow sends, for each $t$, ${\mathcal F}_{0}$ to ${\mathcal F}(t)$. 
Finally, we let 
\begin{equation} \label{DefT}
{\mathcal T}(w):=\omega_{0} \circ \eta(t,\cdot)^{-1} \circ \varphi(t).
\end{equation}
It is trivial that ${\mathcal T}({\mathcal C}) \subset {\mathcal C}$. It remains to prove that ${\mathcal T}$ is continuous and that ${\mathcal T}({\mathcal C})$ is relatively compact in $L^{\infty}((0,T); L^{3}({\mathcal F}_{0}))$. \par
\ \par
Let us begin with the continuity. We consider $(w_{n}) \in {\mathcal C}^{\N}$ converging to  $w \in {\mathcal C}$ for the $L^{\infty}(0,T;L^{3}({\mathcal F}_{0}))$ norm. We associate $u_{n}$ and $\eta_{n}$ corresponding to $w_{n}$ in the above construction, and accordingly $u$ and $\eta$ corresponding to $w$. 
Using \eqref{Eq:vtransporte}, Sobolev imbeddings and Lemma \ref{LemEstLp}, it is not difficult to see that the velocities $u_{n}$ converge to the velocity $u$ in $L^{\infty}(0,T , L^{\infty}({\mathcal F}(t)))$. Also, from Lemma \ref{LemEstLL} we deduce that for some $C>0$,
\begin{equation*}
\| u \|_{L^{\infty}(0,T; \mathcal{LL}({\mathcal F}(t)))}, \  \| u_{n} \|_{L^{\infty}(0,T; \mathcal{LL}({\mathcal F}(t)))} \leq C.
\end{equation*}
This involves the uniform convergence of $\eta_{n}^{-1}$ to $\eta^{-1}$. The convergence of $\varphi_{n}$ to $\varphi$ comes from the continuity of $\Psi$. 
From this we can deduce that
\begin{equation*}
{\mathcal T}(w_{n}) \longrightarrow {\mathcal T}(w) \ \text{ in } \ L^{\infty}((0,T);L^{3}({\mathcal F}_{0})).
\end{equation*}
Indeed, if $\omega_{0} \in C^{0}(\overline{{\mathcal F}_{0}})$, this can be straightforwardly deduced from the uniform continuity of $\omega_{0}$ and \eqref{DefT}. The general case can be inferred by using the density of $C^{0}(\overline{{\mathcal F}_{0}})$ in $L^{\infty}({\mathcal F}_{0})$ for the $L^{3}({\mathcal F}_{0})$ topology. \par
\ \par
Now let us prove the relative compactness of ${\mathcal T}({\mathcal C})$ in $L^{\infty}(0,T; L^{3}({\mathcal F}_{0}))$. This is a consequence of the following lemma.
\begin{Lemma} \label{Lem:DelaCompacite}
Let $C>0$, $\alpha \in (0,1)$ and $\omega_{0} \in L^{\infty}({\mathcal F}_{0})$. Then the set
\begin{equation*}
A(\omega_{0}):=\Big\{ \omega_{0} \circ \psi(t,x) \text{ for } \psi \in C^{\alpha}([0,T] \times {\mathcal F}_{0};{\mathcal F}_{0}) \text{ such that } 
\| \psi \|_{C^{\alpha}} \leq C \text{ and } \psi \text{ is measure-preserving} \Big\},
\end{equation*}
is relatively compact in $L^{\infty}(0,T;L^{3}({\mathcal F}_{0}))$.
\end{Lemma}
\begin{proof}[Proof of Lemma \ref{Lem:DelaCompacite}]
We prove the total boundedness of $A(\omega_{0})$. Let us be given $\varepsilon>0$. There exists $\omega_{1} \in C^{0}(\overline{{\mathcal F}_{0}})$ such that
\begin{equation*}
\| \omega_{1} - \omega_{0} \|_{L^{3}({\mathcal F}_{0})} \leq \varepsilon.
\end{equation*}
Due to the continuity of $\omega_{1}$, it is a direct consequence of Ascoli's theorem that $A(\omega_{1})$ is relatively compact in $C^{0}([0,T] \times \overline{{\mathcal F}_{0}})$, and hence in $L^{\infty}(0,T;L^{3}({\mathcal F}_{0}))$. We deduce the existence of $\psi_{1}$, \dots, $\psi_{N}$ as above such that
\begin{equation*}
A(\omega_{1}) \subset B(\omega_{1} \circ \psi_{1};\varepsilon) \cup \cdots \cup B(\omega_{1} \circ \psi_{N};\varepsilon),
\end{equation*}
where the balls are considered in the space $L^{\infty}(0,T;L^{3}({\mathcal F}_{0}))$. One sees that
\begin{equation*}
A(\omega_{0}) \subset B(\omega_{1} \circ \psi_{1};2\varepsilon) \cup \cdots \cup B(\omega_{1} \circ \psi_{N};2\varepsilon) ,
\end{equation*}
which concludes the proof of the lemma.
\end{proof}
\begin{proof}[Back to the proof of Proposition \ref{SolideImpose}]
Using again Lemma \ref{LemEstLL}, we see that we have uniform log-Lipschitz estimates on the velocities $u$ as $w \in {\mathcal C}$. This implies uniform H\"older estimates on the flows $\eta$ that we constructed for $w \in {\mathcal C}$. So we conclude by Lemma \ref{Lem:DelaCompacite} that ${\mathcal T}({\mathcal C})$ is relatively compact. \par
Hence we deduce by Schauder's fixed point theorem that ${\mathcal T}$ admits a fixed point. One checks easily that the corresponding $u$ fulfills the claims. \par
Finally, the uniqueness is proved exactly as in Yudovich's original setting for the fluid alone; alternatively, one can use the proof of uniqueness established in Subsection \ref{section:UniqYudo} with $(\ell_{1},r_{1})=(\ell_{2},r_{2})$. \par
\end{proof}
\subsection{Continuous dependence on the solid movement}
Now we prove that the solution constructed in Subsection \ref{Subsec:PrescribedMovement} depends continuously on the solid movement $(\ell,r)$. \par
Precisely, given $T>0$ and $(\ell,r)$ as in Subsection \ref{Subsec:PrescribedMovement}, we denote
$u^{\ell,r}$ the unique corresponding solution $u$ given by Proposition \ref{SolideImpose}, and $\omega^{\ell,r}:=\curl u^{\ell,r}$ the corresponding vorticity. We associate then
\begin{equation*}
\tilde{u}^{\ell,r} := {u}^{\ell,r} \circ \varphi^{\ell,r}
\ \text{ and } \
\tilde{\omega}_{n} := \omega^{\ell,r} \circ \varphi^{\ell,r}.
\end{equation*}
We have the following Proposition.
\begin{Proposition} \label{DependanceContinue}
Let $T>0$, $(\ell_{n},r_{n}) \in C^{0}([0,T]; \R^{2} \times \R)^{\N}$ and  $(\ell,r) \in C^{0}([0,T]; \R^{2} \times \R)$ such that $(\ell_{n},r_{n})$ and $(\ell,r)$ satisfy \eqref{LoinDuBord} and
\begin{equation*}
(\ell_{n},r_{n}) \longrightarrow (\ell,r) \text{ in } C^{0}([0,T]; \R^{2} \times \R) \ \text{ as } \ n \rightarrow +\infty.
\end{equation*}
Then
\begin{equation} \label{CVTildeu}
\tilde{u}^{\ell_{n},r_{n}}\longrightarrow \tilde{u}^{\ell,r} \text{ in } C^{0}([0,T] \times \overline{{\mathcal F}_{0}}) \ \text{ as } \ n \rightarrow +\infty.
\end{equation}
\end{Proposition}
\begin{proof}[Proof of Proposition \ref{DependanceContinue}]
Following Subsection \ref{Subsec:PrescribedMovement}, we see that $(\tilde{u}^{\ell_{n},r_{n}})$ is relatively compact in 
$C^{0}([0,T] \times \overline{{\mathcal F}_{0}})$. Also, the sequence $(\tilde{\omega}^{\ell_{n},r_{n}})$ is weakly-$*$ relatively compact in $L^{\infty}((0,T) \times {\mathcal F}_{0})$. To prove \eqref{CVTildeu}, it is hence sufficient to prove that the unique limit point of the sequence $(\tilde{u}^{\ell_{n},r_{n}}, \tilde{\omega}^{\ell_{n},r_{n}})$ in the space $C^{0}([0,T] \times \overline{{\mathcal F}_{0}}) \times [L^{\infty}((0,T) \times {\mathcal F}_{0})-w*]$ is $(\tilde{u}^{\ell,r},\tilde{\omega}^{\ell,r})$. \par
Now consider a converging subsequence of $(\tilde{u}^{\ell_{n},r_{n}},\tilde{\omega}^{\ell_{n},r_{n}})$. For notational convenience, we still denote this subsequence $(\tilde{u}^{n},\tilde{\omega}^{n} )$, and call $(\tilde{\overline{u}},\tilde{\overline{\omega}})$ the limit. We associate the functions $w^{n}$, $\eta^{n}$, $\overline{w}$, $\overline{\eta}$ corresponding to $(\ell_{n},r_{n},\tilde{u}^{n},\tilde{\omega}^{n})$ and $(\ell,r,\tilde{\overline{u}},\tilde{\overline{\omega}})$
as in Subsection \ref{Subsec:PrescribedMovement}. \par
By uniqueness in Proposition \ref{SolideImpose}, to prove
\begin{equation*}
(\tilde{\overline{u}},\tilde{\overline{\omega}}) = (\tilde{u}^{\ell,r},\tilde{\omega}^{\ell,r}),
\end{equation*}
is it sufficient to prove that $(\tilde{\overline{u}},\tilde{\overline{\omega}})$ corresponds to a solution of Proposition \ref{SolideImpose} with prescribed solid movement $(\ell,r)$. For that, we observe that the relation
\begin{equation*}
w^{n}(t,\cdot) = \omega_{0} \circ \eta^{n}(t,\cdot)^{-1} \circ \varphi^{\ell_{n},r_{n}}(t).
\end{equation*}
passes to the $L^{\infty}((0,T) \times {\mathcal F}_{0})$ weak-$*$ limit (or to the $L^{\infty}(0,T;L^{p}({\mathcal F}_{0}))$ one, $p \in [1,\infty)$) since $\tilde{u}^{\ell_{n},r_{n}}$ converges uniformly to $\tilde{\overline{u}}$, so the corresponding flows converge uniformly (using again the uniform log-Lipschitz estimates). So we infer
\begin{equation*}
\overline{w}(t,\cdot) = \omega_{0} \circ \overline{\eta}(t,\cdot)^{-1} \circ \varphi^{\ell,r}(t).
\end{equation*}
On the other side, it is not difficult to pass to the limit in \eqref{Eq:vtransporte}, so that in particular
\begin{equation*}
\curl \Big( \tilde{\overline{u}}\circ ( \varphi^{\ell,r}(t,\cdot)^{-1} )  \Big) =  \overline{w}(t,\cdot) \circ ( \varphi^{\ell,r}(t,\cdot)^{-1} ).
\end{equation*}
Hence we deduce that $(\tilde{\overline{u}},\tilde{\overline{\omega}})$ is indeed a solution of Proposition \ref{SolideImpose} with solid movement $(\ell,r)$. The convergence \eqref{CVTildeu} follows. %\par
\end{proof}

\subsection{Endgame}

Let us now proceed to the proof of Theorem \ref{ThmYudo}.
Once again we are going to use Schauder's fixed point theorem.

 Let $d := d \left( {\mathcal S}_{0} , \partial \Omega \right) > 0$ and $C>0$.
 We introduce
\begin{equation*}
{\mathcal D}:= \Big\{ (\ell,r) \in C^{0}([0,T];\R^{3}) \ \Big/  \  \| (\ell,r)  \|_{C^{0}([0,T];\R^{3})} \leq C \Big\} ,
\end{equation*}
where $T >0$ is chosen such that $ CT (1 + \diam ({\mathcal S}_{0}) ) \leq    \frac{{d}}{3} $.
Observe that this condition yields in particular that $\tau^{\ell,r}$ satisfies 
\begin{equation} \label{loinbord}
\dist \left( \tau^{\ell,r}(t)({\mathcal S}_{0}), \partial \Omega \right)  \geq \frac{{d}}{3} \ \text{ for any } \ t \in  [0,T].
\end{equation}
Note that ${\mathcal D}$ is closed and convex. \par
Now we construct an operator ${\mathcal A}$ on ${\mathcal D}$ in the following way. To $(\ell,r) \in {\mathcal D}$, we associate, as in Subsection \ref{Subsec:PrescribedMovement}, $Q(t)$, ${\mathcal S}(t)$, ${\mathcal F}(t)$ and $u$ as the fixed point of the operator ${\mathcal T}^{\ell,r}$. 
We also consider the  Kirchhoff potentials $\Phi_{i}$ as in   \eqref{KirchoffPotentials} and 
the mass matrix $\mathcal{M}$ as in \eqref{Added}.
Then, we define  ${\mathcal A}(\ell,r):=(\tilde{\ell},\tilde{r})$, where  for any $ t \in  [0,T]$,
\begin{equation*} 
\begin{bmatrix} \tilde{\ell} \, \\ \tilde{r} \end{bmatrix} (t)
= 
\begin{bmatrix} \ell_0 \\ {r}_0 \end{bmatrix}
+  \int_0^t  \mathcal{M}^{-1} (s)
\left(
\begin{bmatrix}  \displaystyle\int_{ \mathcal{F}(s)}  u\cdot[  (u \cdot \nabla) \nabla  \Phi_i ] \, dx   \end{bmatrix}_{i \in \{1,2,3\}} -  \begin{bmatrix}  B_i (s) \end{bmatrix}_{i \in \{1,2,3\}}
 \right) \, ds ,
\end{equation*}
with
\begin{equation*}
 B_i  (s) := \begin{bmatrix}\ell  \\ r  \end{bmatrix} 
 \cdot \displaystyle\int_{\partial \mathcal{S}(s)} ( u\cdot  \nabla  \Phi_i )  \begin{bmatrix}n  \\  (x-h(s))^\perp \cdot n  \end{bmatrix} \, d \sigma .
\end{equation*}
Due to the boundedness of $\| \nabla \Phi_{i} \|_{C^{1,\alpha}}$, ${\mathcal M}^{-1}$ under the condition \eqref{loinbord} and the one of $\| u \|_{\infty}$,
shrinking $T$ if necessary, we have that  ${\mathcal A}({\mathcal D}) \subset {\mathcal D}$. \par
Now, let us prove that ${\mathcal A}$ has a fixed point in ${\mathcal D}$.
That ${\mathcal A}({\mathcal D})$ is relatively compact in $C^{0}([0,T];\R^{3})$ follows from Ascoli's theorem.
That ${\mathcal A}$ is continuous follows from Proposition \ref{DependanceContinue} and the convergence for all $t$, under the assumptions of Proposition \ref{DependanceContinue}:
\begin{equation*}
\Phi_{i}^{\ell_{n},r_{n}} \circ \varphi^{\ell_{n},r_{n}} \longrightarrow \Phi_{i}^{\ell,r} \circ \varphi^{\ell,r} \ \text{ in } \ C^{2}(\overline{{\mathcal F}_{0}}) \ \text{ as } \ n \rightarrow +\infty.
\end{equation*}
This convergence can be deduced from the compactness of the sequence $(\nabla \Phi_{i}^{\ell_{n},r_{n}})$ in $C^{1}(\overline{{\mathcal F}_{0}})$ (due to Lemma \ref{LemHolder}) and the fact that $\Phi^{\ell,r} \circ \varphi^{\ell_{n},r_{n}} \circ (\varphi^{\ell,r})^{-1}$ converges to a function satisfying the correct system \eqref{KirchoffPotentials} in the limit (use for instance the computations of Subsection \ref{section:UniqYudo}).
Therefore we can apply Schauder's theorem which proves the existence of a fixed point. \par
To see that a fixed point of the operator  ${\mathcal A}$ corresponds to  a solution of \eqref{vietendue}-\eqref{Solide2} it is sufficient to observe that, thanks to an integration by parts, the solid equations can be recast as 
\begin{equation*} 
\mathcal{M} \begin{bmatrix} \ell \\ r \end{bmatrix}' 
= \begin{bmatrix}  \displaystyle\int_{ \mathcal{F}(t)}  u\cdot[  (u \cdot \nabla) \nabla  \Phi_i ] \, dx 
\end{bmatrix}_{i \in \{1,2,3\}} -  \begin{bmatrix}  B_i \end{bmatrix}_{i \in \{1,2,3\}}  .
\end{equation*}
Observe that this reformulation is slightly different from  \eqref{EvoMatrice} and is obtained by using \eqref{Euler1} and an integration by parts instead of \eqref{DecompP}. This establishes the existence part of Theorem \ref{ThmYudo}. \par
\ \par
Finally, that the lifetime can be  uniquely limited by a possible encounter of the body with the boundary follows by
contraposition, as a positive distance allows to extend the solution for a while, according to the previous arguments.
\par
\ \par
\begin{Remark}
Mixing the techniques of \cite{GSweak} and  the ones of \cite{ogfstt}, one could prove some results about the regularity in time of the flows associated to the solutions given by Theorem \ref{ThmYudo}.
More precisely, consider a solution $(\ell,r,u)$ given by Theorem \ref{ThmYudo}, then the corresponding fluid velocity field $u$ is log-Lipschitz in the $x$-variable; consequently there exists a unique flow map $\eta$ continuous from  $\R  \times{\mathcal F}_0 $ to $\mathcal{F}(t) $ such that
\begin{equation*}
\eta (t,x)  =  x + \int^{t}_0 u (s, \eta (s,x) ) ds .
\end{equation*}
Moreover there exists $c >0$ such that for any $t$, the vector field  $ \eta(t,\cdot) $ lies in the  H\"older space 
\begin{equation*}
C^{0 , \exp(-c|t|  \| \omega_0  \|_{L^{\infty}  ({\mathcal F}_0   )})}  ({\mathcal F}_0   ).
\end{equation*}
If one assume that the boundaries $\partial {\mathcal S}_{0}$ and $\partial \Omega$  are $C^{k+1,\nu}$, with $k  \in \N$ and $\nu \in (0,1)$, 
then the flow $( \tau , \eta )$ are   $C^k$ from $[0,T]$ to $ SE(2) \times C^{0 , \exp(-c T  \| \omega_0  \|_{L^{\infty}  ({\mathcal F}_0   )} )}  ({\mathcal F}_0   )$.
If one assume that the boundaries $\partial {\mathcal S}_{0}$ and $\partial \Omega$  are Gevrey of order $M \geq 1$, then  the flow $( \tau , \eta )$ are   Gevrey of order $M+2$
from $[0,T]$ to $ SE(2) \times C^{0 , \exp(-c T  \| \omega_0  \|_{L^{\infty}  ({\mathcal F}_0   )} )}  ({\mathcal F}_0   )$.
In particular in the case where  $M=1$, we see that when the boundaries are real-analytic, then the flows $( \tau , \eta )$ belong to the Gevrey space ${\mathcal G}^{3}$.
\end{Remark}
\par
\ \par
\noindent
{\bf Acknowledgements.} The  authors are partially supported by the Agence Nationale de la Recherche, Project CISIFS,  grant ANR-09-BLAN-0213-02. 
%They thank M. Geissert, K. G{\"o}tze and M. Hieber for providing them a preprint of their paper \cite{GGH}.
%
%
%

\end{document}